\documentclass[12pt]{amsart}
\usepackage{amsmath}
\usepackage{amsthm}
\usepackage{amssymb}
\usepackage{verbatim}

\usepackage{subfigure}
\usepackage{verbatim}
\usepackage{url}
\numberwithin{equation}{section}
\usepackage{textcomp}

\usepackage{xy}
\usepackage{color}
\usepackage[pdftex]{graphicx}
\usepackage{epstopdf}
\usepackage{mathrsfs}
\usepackage{enumerate}
\usepackage{hyperref}
\usepackage{color}

\usepackage[dvipsnames]{xcolor}

\def\displayandname#1{\rlap{$\displaystyle\csname #1\endcsname$}%
                      \qquad \texttt{\char92 #1}}

\makeatletter
\def\url@leostyle{%
  \@ifundefined{selectfont}{\def\UrlFont{\sf}}{\def\UrlFont{\small\ttfamily}}}
\makeatother

\DeclareMathAlphabet{\mathbbold}{U}{bbold}{m}{n}
\newcommand{\zero}{\mathbbold{0}}
\newcommand{\unit}{\mathbbold{1}}

\newcommand{\per}{\operatorname{per}}

\newcommand{\Id}{\operatorname{Id}}

\newcommand\ap{{\operatorname{ap}}}

\newcommand\adj{{\operatorname{adj}}}
\newcommand{\rmax}{\mathbb{R}_{\max}}
\newcommand{\R}{\mathbb{R}}

\newcommand{\rcong}{\mathbin{\rotatebox[origin=c]{-90}{$\cong$}}}
\newcommand{\lin}{\mathbin{\rotatebox[origin=c]{90}{$\in$}}}
\newcommand{\lfeq}{\mathbin{\rotatebox[origin=c]{90}{$=$}}}

\newcommand{\lleq}{\mathbin{\rotatebox[origin=c]{-90}{$\leq$}}}

\newtheorem{thm}{Theorem}[section]
\newtheorem{pro}[thm]{Proposition}
\newtheorem{lem}[thm]{Lemma}

\newtheorem{cor}[thm]{Corollary}

\newtheorem{obs}[thm]{Observation}

\newtheorem{alg}{Algorithm}

\theoremstyle{definition}
\newtheorem{df}[thm]{Definition}

\theoremstyle{remark}
\newtheorem{rem}[thm]{Remark}
\newtheorem{exa}[thm]{Example}

\newtheorem{nota}[thm]{Notation}

\usepackage[colorinlistoftodos,bordercolor=orange,backgroundcolor=orange!20,linecolor=orange,textsize=scriptsize]{todonotes}
\usetikzlibrary{arrows}
\title{Optimal assignments with supervisions}

\author{{A}di Niv}
\address{Adi Niv,
Mathematics Department, Science Faculty, Kibbutzim College.
Address: Kibbutzim College, 149 Namir Rd., Tel-Aviv, Israel.}
\email{adi.niv@smkb.ac.il}

\author{{M}arie MacCaig}
\address{Marie MacCaig,
CMAP, \'Ecole  Polytechnique.
Address:  Route de Saclay,~91128 ${\ \ }$ Palaiseau Cedex, France.}
\email{m.maccaig.maths@outlook.com}

\author{{S}erge\u{\i} Sergeev}
\address{Serge\u{\i} Sergeev,
School of Mathematics, University of Birmingham.
Address: Watson Building, University of Birmingham, Edgbaston Birmingham B15 2TT UK.}
\email{s.sergeev@bham.ac.uk}

\thanks{The first author was supported by INRIA postdoctoral fellowship and the Chateaubriand grant.}
\thanks{The second author was supported by a public grant as part of the \emph{Investissement d'avenir} project, reference ANR-11-LABX-0056-LMH, LabEx LHM.} 
\thanks{The third author was supported by EPSRC Grant EP/P019676/1.}
\begin{document}
  
\begin{abstract} 
In this paper we provide a new graph theoretic proof of the tropical Jacobi identity, recently obtained in~\cite{AGN}. 
We also develop an application of this theorem to optimal assignments with supervisions. That is,
 optimally    assigning  multiple tasks  to one team, or daily tasks
to multiple teams, where  each team has a supervisor task or a supervised task. 

\vskip 0.15 truecm

\vskip 0.1 truecm

\noindent \textit{AMSC: 05C17; 05C22; 05C38; 05C50; 05E15; 15A15; 15A24; 15A80; 90B80.} 	

\end{abstract}

\maketitle

\thispagestyle{empty}

%%%%%%%%%%%%%%%%%%%%%%%%%%%%%%%%%%%%%%%%%%%%%%%%%%%%%

\section{Introduction} 
The tropical semiring $\R_{\max}$ is the set $\R\cup\{-\infty\}$ of real numbers  formally joined with~$-\infty$, equipped with the additive operation $a\oplus b=\max\{a,b\}$  
and the multiplicative operation $a\odot b=a+b$ (for all $a,b\in\R_{\max}$). 
In this language, 
 the  \textit{tropical permanent} of a matrix $A\in\R_{\max}^{n\times n}$ is 
$$\per(A)=\bigoplus_{\pi\in S_n}\ \bigodot_{i\in[n]}A_{i,\pi(i)}=\max_{\pi\in S_n}\ \sum_{i\in[n]}A_{i,\pi(i)},$$ and $\sigma\in S_n$ is an \textit{optimal permutation} if 
$\per(A)=\bigodot_{i\in[n]}A_{i,\sigma(i)}=\sum_{i\in[n]}A_{i,\sigma(i)}$. Here $[n]=\{1,\ldots,n\}$ and 
$S_n$ is the set of all permutations on $[n]$.

For~$A\in\rmax^{n\times n}$ we
 denote by~$A[I,J]$ or~$A_{I, J}, \text{where}\ I\subset [n], J\subset [n],$  the~$|I|\times |J|$ \textit{submatrix}  of $A$ with rows in~$I$ and columns in~$J$. Given $I\subseteq [n]$ we denote by $I^c$ the complement
of $I$ (so that $I\cup I^c=[n]$ and $I\cap I^c=\emptyset$).

The \textit{tropical adjoint} $\adj(A)$ of $A\in\R_{\max}^{n\times n}$ is defined by 
$\adj(A)_{i,j}=\per(A_{\{j\}^c,\{i\}^c}),$
i.e.~$\big(\sum_{i\in[n]}A_{i,\pi(i)}\big)-A_{j,i}$ for some permutation~$\pi\in S_n$ such that $\pi(j)=i$ and 
such that it is optimal among all such permutations.

Motivated by  Butkovic's combinatorial interpretations of various objects in tropical algebra (see~\cite{MA}), and in particular, since optimal permutations are  associated to the optimal assignment problem, we ask whether there exists an interpretation to tropical identities, focusing on Jacobi identity, using some form of `partial assignment problem' and/or `multiple assignment problem'.

The tropical Jacobi identity of a matrix $A$, if $\per(A)=0$, states that the permanent 

of a $k\times k$ submatrix of $\adj(A)$ is either equal to or surpasses the  permanent of the corresponding $(n-k)\times (n-k)$-submatrix of $A$.  This tropical identity was obtained in~\cite{AGN}, motivated by the classical Jacobi identity described in~\cite[Section~1.2]{Fallat&Johnson}.

The digraph associated with the permanent of~$A$ describes the well known optimal assignment problem in the sense that 
$\per(A)$ provides the weight of  optimal permutations in the digraph associated with~$A$.
When we consider the digraph associated with~$ \adj(A)_{i,j}$, we are optimizing over permutations in~$S_n$ with a single edge removed. 
In this language, we provide a graph theoretic proof of a tropical analogue of the Jacobi identity,  by showing that the weight of an optimal permutation of the digraph associated with
 a $k\times k$ submatrix of $\adj(A)$ either equals to the
weight of the  optimal permutation of the digraph associated with the corresponding $(n-k)\times (n-k)$-submatrix of $A$, or there exist at least two such optimal permutations.

 Motivated by this graph theoretic proof, we study the optimal assignment problem under a  given condition or requirement that a person~$i$ performs a fixed
 assignment~$j$, 
whose cost/profit is out of consideration and would later be considered as the supervisor assignment. 

We develop applications to team assignments, where supervision needs to take place. 
Thus we show that the Jacobi identity is closely related to optimizing  multiple tasks, involving    
multiple teams, daily assignments, supervisor assignment or supervised assignments.

%%%%%%%%%%%%%%%%%%%%%%%%%%%%%%%%%%%%%%%%%%%%%%%%%%%%%

\section{Preliminaries}

We provide some known  graph theory definitions, as well as define ways to represent multiple assignments as series of perfect matchings.

\subsection{Basic definitions}
\begin{df} 
A \textit{digraph} (or  {directed graph}) is an ordered pair $G=(V_G,E_G)$ where
$V_G$  is a set whose elements are called \textit{nodes} (or vertices),
and $E_G$ is a set of ordered pairs of vertices, called \textit{directed edges} (or  arcs),
allowing loops and multiple edges.

%A \textit{digraph}~$G=(V_G,E_G)$ is a directed graph.  That is a graph with a set $E_G$ of directed edges on $|V_G|$ nodes,  allowing loops and multiple edges. 
The number of edges terminating (resp.~originating) at~$v\in V_G$ is denoted by~$\deg^-(v)$ (resp.~$\deg^+(v)$). 
The \textit{source} (resp.~\textit{target}) of an edge $e\in E_G$ is denoted by $s(e)$ (resp.~$t(e)$). The edge $e$ may 
also be denoted by $(v_i,v_j) $ if $v_i=s(e), v_j=t(e)$.  A  graph  is called \textit{simple} if it does not have multiple edges.

A \textit{multigraph} is a (di)graph which is permitted to have multiple edges that have the same end nodes. 
Thus two vertices may be connected by more than one edge.
We say two multigraphs are equal if they have the same edge set, counting multiplicities.
\end{df}
\begin{df} A \textit{bipartite graph}~$H=(V_{H,1},V_{H,2},E_H)$ is a nondirected graph such that $i\in V_{H,1}$ if and only if $j\in V_{H,2}$  for every $(i,j)\in E_H$. The number of edges exiting $v\in V_{H,1}\cup V_{H,2}$ is denoted by $\deg(v)$. We say $H$ is \textit{equally partitioned} if $|V_{H,1}|=|V_{H,2}|$.
The \textit{complete bipartite graph}, denoted $K_{m, n}$ is the bipartite graph $$G=([m], [n], E):\ E=\{(u,v): u\in [m], v\in [n]\}.$$
A \textit{star} is the complete bipartite graph $K_{1,k}$, denoted as $ST_k$.
\end{df}
\begin{df} A graph (resp.~digraph, and in particular multigraph) $G=(V, E)$ is \textit{$k$-regular} if $\deg(v)=k $
 ({resp.~}$\deg^{+}(v)=\deg^{-}(v)=k),\ \forall v\in V. $

\end{df}

\begin{df} A \textit{path} in a digraph $G=(V, E)$ is a sequence of nodes and edges, $P=(v_1, e_1, v_2, e_2,\dots, e_{n-1}, v_{n})$ such that, for all $i\in [n-1]$, $e_i=(v_i,v_{i+1})$. In particular, $s(P)=s(e_1)=v_1,\ t(P)=t(e_{n-1})=v_n$.
If $v_1=v_n$, then $P$ is \textit{closed}.
If $P$ is a path in which all intermediate nodes are distinct, and different from its source and target, then $P$ is \textit{elementary} and denoted by $(v_1, v_2,\dots, v_{n})$ (when clear which edges are used).  The \textit{length of a path} $\ell(P)=n-1$ is the number of its edges. 
A \textit{cycle} is an elementary closed path, denoted by $(v_1, v_2,\dots, v_{n-1}, v_1)$.  
\end{df}

\begin{df} Let $G=(V,E)$ be a graph and $E'\subseteq E$. The \textit{subgraph of $G$ induced by $E'\subseteq E$} is the subgraph  $G'=(V(E'),E')$, where $V(E')\subseteq V$ denotes the set of endpoints (sources and targets when $G$ is a digraph) of $ E'$.
\end{df}

\begin{df} If $G=(V_G, E_G)$ and $H=(V_H, G_H)$, then $G+H=(V_G\uplus V_H, E_G\uplus E_H)$ is called the \emph{disjoint union graph}.
The graph $kG$ is formed of $k$ disjoint copies of $G$.
\end{df}

\begin{df} We say $G=(V,E)$ is a  \textit{$k$-bipartite}
graph if $$E=\biguplus_{r=1}^k E_r ,$$  where $B_r=(U_r, V_r, E_r),\ U_r, V_r\subseteq V,\ \forall r\in[k]$ are bipartite graphs.
We say $G$ is \textit{equally 
partitioned} if $|U_1|=\dots=|U_k|=|V_1|=\dots=|V_k|.$

The graph~$G$ is \textit{disjoint-$k$-bipartite} 
if~$G=B_1+\dots+B_k$. 
The graph~$G$ is  \textit{star-$k$-bipartite}  if~$\{B_r\}_{r\in[k]}$ are
   \emph{glued} at a common vertex set:~$V_r=V,\ \forall r\in[k]$, 
and~$G$ may be denoted by~$(U_1,\dots, U_k, V, E)$.
The graph~$G$ is  \textit{path-$k$-bipartite}  if~$\{B_r\}_{r\in[k]}$ are
   \emph{concatenated}:~$V_r=U_{r+1},\ \forall r\in[k-1]$, 
and~$G$ may be denoted by~$(U_1,\dots U_{k+1}, E)$.

The number of edges exiting~$v\in V_{r}$ towards $V_{r-1}$ (resp.~$V_{r+1}$)  is denoted by~$\deg_t(v)$ (resp.~$\deg_s(v)$). 
\end{df}

\begin{lem}\label{eqv1}
 There exists a one-to-one correspondence   from the set of digraphs with $n$ nodes onto  the set of equally partitioned bipartite graphs with $2n$ nodes.

\end{lem} 

\begin{proof} Let $(V_G,E_G)$  be a digraph, and $(V_{H,1},V_{H,2},E_H)$ be an equally partitioned bipartite graph.
We define the correspondence by duplicating the set $V_G$: $v_i\in V_{H,1}, u_i\in V_{H,2},\ \forall i\in V_G,$
which function as the sets of sources and targets respectively. That is,  $(v_i,u_j)\in E_H,\ \forall (i,j)\in E_G.$
This correspondence is of course  one-to-one and onto.
\end{proof}

\begin{lem}There exists a one-to-one correspondence   between all equally partitioned  $k$-bipartite graphs.
\end{lem}
\begin{proof} Straightforward, using their $B_r$-presentations.\end{proof}
             
\begin{exa} See Figure~\ref{crsp}. The bold edges correspond to the digraph on top and will be recalled later on.

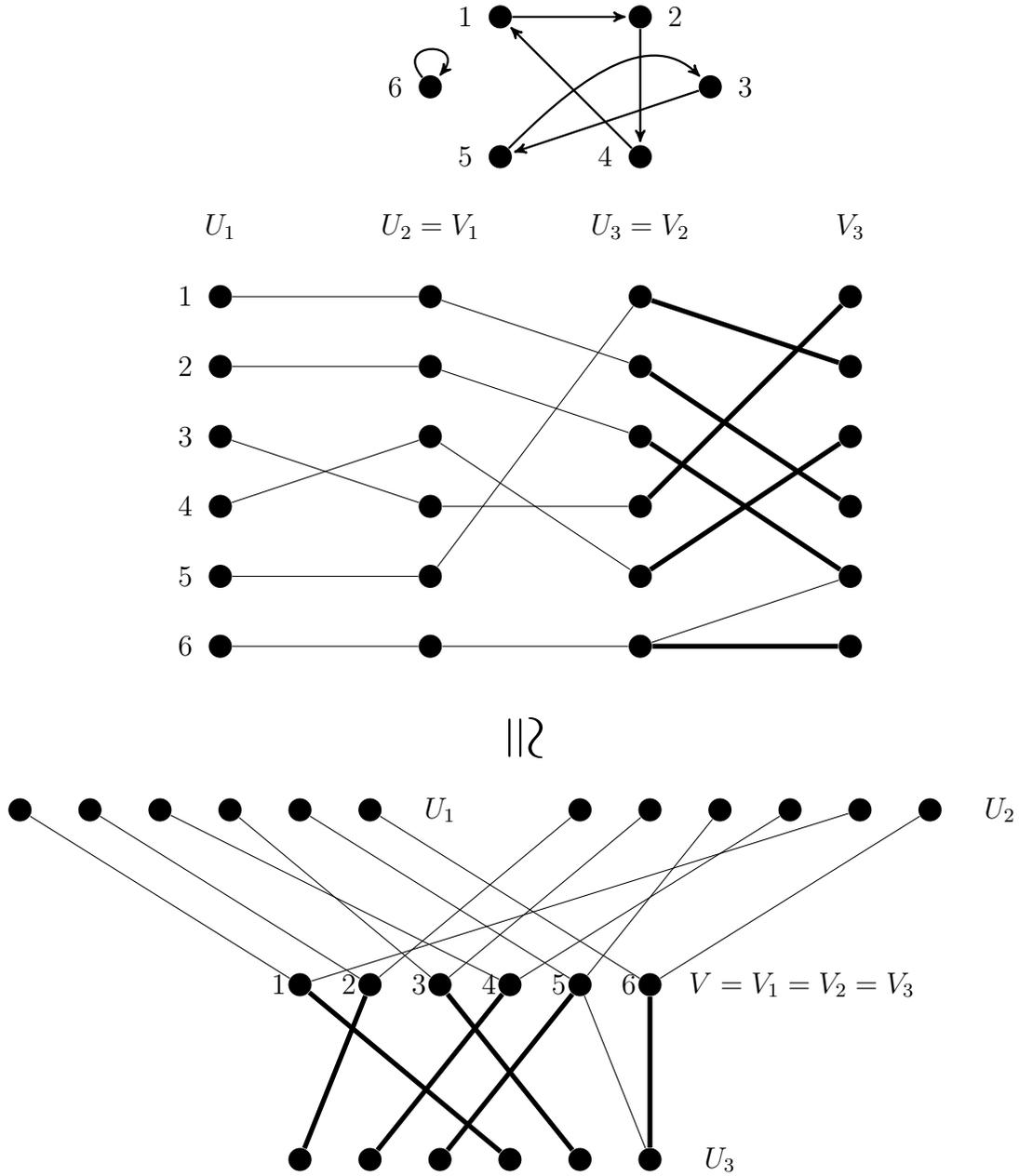
\begin{figure}[h]
\begin{center}
\begin{tikzpicture}[main_node/.style={circle,fill=black,minimum size=0.8em,inner sep=1pt]},
sx_node/.style={->,>=stealth',shorten >=1pt,circle,minimum size=0.8em,inner sep=1pt,line width=0.3mm}]

\node[draw=black,circle,white](u1) at (0,6) {$\color{black}U_1$};   \node[draw=black,circle,white](u2) at (3,6) {$\color{black}U_2=V_1$};   \node[draw=black,circle,white](u3) at (6,6) {$\color{black}U_3=V_2$};\node[draw=black,circle,white](v) at (9,6) {$\color{black}V_3$};
\node[draw=black,circle,white](1) at (-0.5,5) {$\color{black}1$};       \node[main_node] (11) at (0,5) {};          \node[main_node] (21) at (3,5) {};        \node[main_node] (31) at (6,5) {};               \node[main_node] (41) at (9,5) {};  
\node[draw=black,circle,white](2) at (-0.5,4) {$\color{black}2$};    \node[main_node] (12) at (0,4) {};      \node[main_node] (22) at (3,4) {};        \node[main_node] (32) at (6,4) {};      \node[main_node] (42) at (9,4) {};  
\node[draw=black,circle,white](3) at (-0.5,3) {$\color{black}3$};     \node[main_node] (13) at (0,3) {};      \node[main_node] (23) at (3,3) {};             \node[main_node] (33) at (6,3) {};               \node[main_node] (43) at (9,3) {};  
\node[draw=black,circle,white](4) at (-0.5,2) {$\color{black}4$};     \node[main_node] (14) at (0,2) {};       \node[main_node] (24) at (3,2) {};             \node[main_node] (34) at (6,2) {};               \node[main_node] (44) at (9,2) {};  
\node[draw=black,circle,white](5) at (-0.5,1) {$\color{black}5$};   \node[main_node] (15) at (0,1) {};         \node[main_node] (25) at (3,1) {};             \node[main_node] (35) at (6,1) {};               \node[main_node] (45) at (9,1) {};  
 \node[draw=black,circle,white](6) at (-0.5,0) {$\color{black}6$};    \node[main_node] (16) at (0,0) {};      \node[main_node] (26) at (3,0) {};             \node[main_node] (36) at (6,0) {};               \node[main_node] (46) at (9,0) {};  

 \draw[main_node]  (11) edge node{} (21);      \draw[main_node]  (21) edge node{} (32);       \draw[line width=0.7mm]  (31) -- (42);   
  \draw[main_node]  (12) edge node{} (22);      \draw[main_node]  (22) edge node{} (33);       \draw[line width=0.7mm]  (32) -- (44);   
 \draw[main_node]  (13) edge node{} (24);      \draw[main_node]  (23) edge node{} (35);          \draw[line width=0.7mm]  (33) -- (45);   
 \draw[main_node]  (14) edge node{} (23);      \draw[main_node]  (24) edge node{} (34);          \draw[line width=0.7mm]  (34) -- (41);   
 \draw[main_node]  (15) edge node{} (25);      \draw[main_node]  (25) edge node{} (31);          \draw[line width=0.7mm]  (35) -- (43);   
 \draw[main_node]  (16) edge node{} (26);      \draw[main_node]  (26) edge node{} (36);          \draw[line width=0.7mm]  (36) -- (46);             
\draw[main_node]  (36) edge node{} (45);  

\node[main_node] (101) at (4,9) {}; 
\node[white](201) at (3.5,9) {$\color{black}1$};
\node[main_node] (102) at (6,9) {}; 
\node[white](202) at (6.5,9) {$\color{black}2$};

\node[main_node] (106) at (3,8) {}; 
\node[white](206) at (2.5,8) {$\color{black}6$};
\node[main_node] (103) at (7,8) {}; 
\node[white](203) at (7.5,8) {$\color{black}3$};

\node[main_node] (105) at (4,7) {}; 
\node[white](205) at (3.5,7) {$\color{black}5$};
\node[main_node] (104) at (6,7) {}; 
\node[white](204) at (5.5,7) {$\color{black}4$};

 \draw[sx_node]  (101) to  (102);  
 \draw[sx_node]  (102) to  (104);  
 \draw[sx_node]  (104) to  (101);  
 \draw[sx_node]  (103) to  (105);
 \draw[sx_node]  (105) to [out=45,in=135]  (103);    
 \draw[sx_node]  (106) to [out=130,in=50,looseness=8] (106);

\end{tikzpicture}$$\Huge\rcong$$
\begin{tikzpicture}[main_node/.style={circle,fill=black,minimum size=0.8em,inner sep=1pt]},space_node/.style={circle,white,fill=white,minimum size=0.8em,inner sep=1pt]}]
\node[draw=black,circle,white](u1) at (7,3) {$\color{black}U_1$};
\node[main_node](u11)at(1,3){};\node[main_node](u12)at(2,3){};\node[main_node](u13)at(3,3){};\node[main_node](u14)at(4,3){};\node[main_node](u15)at(5,3){};\node[main_node](u16)at(6,3){};

\node[draw=black,circle,white](u2) at (15,3) {$\color{black}U_2$};
\node[main_node](u21)at(9,3){};\node[main_node](u22)at(10,3){};\node[main_node](u23)at(11,3){};\node[main_node](u24)at(12,3){};\node[main_node](u25)at(13,3){};\node[main_node](u26)at(14,3){};

\node[draw=black,circle,white](V) at (11,0.5) {$\color{black}\ \ \ \ \ \ \ \ \ \ \ \ \ \ \ \ \ V=V_1=V_2=V_3$};
\node[main_node](v1)at(5,0.5){};\node[main_node](v2)at(6,0.5){};\node[main_node](v3)at(7,0.5){};\node[main_node](v4)at(8,0.5){};\node[main_node](v5)at(9,0.5){};\node[main_node](v6)at(10,0.5){};

\node[draw=black,circle,white](V) at (11,-2) {$\color{black}U_3$};
\node[main_node](u31)at(5,-2){};\node[main_node](u32)at(6,-2){};\node[main_node](u33)at(7,-2){};\node[main_node](u34)at(8,-2){};\node[main_node](u35)at(9,-2){};\node[main_node](u36)at(10,-2){};

 \draw[main_node]  (v1) edge node{} (u11);      \draw[main_node]  (v1) edge node{} (u25);       \draw[line width=0.7mm]  (v1) -- (u34);   
 \draw[main_node]  (v2) edge node{} (u12);      \draw[main_node]  (v2) edge node{} (u21);       \draw[line width=0.7mm]  (v2) -- (u31);   
 \draw[main_node]  (v3) edge node{} (u14);      \draw[main_node]  (v3) edge node{} (u22);       \draw[line width=0.7mm]  (v3) -- (u35);   
 \draw[main_node]  (v4) edge node{} (u13);      \draw[main_node]  (v4) edge node{} (u24);       \draw[line width=0.7mm]  (v4) -- (u32);   
 \draw[main_node]  (v5) edge node{} (u15);      \draw[main_node]  (v5) edge node{} (u23);       \draw[line width=0.7mm]  (v5) -- (u33);  
 \draw[main_node]  (v6) edge node{} (u16);      \draw[main_node]  (v6) edge node{} (u26);       \draw[line width=0.7mm]  (v6) -- (u36);     \draw[main_node]  (v5) edge node{} (u36); 

\node(1) at (4.7,0.5) {$\color{black}1$};\node(2) at (5.7,0.5) {$\color{black}2$};\node(3) at (6.7,0.5) {$\color{black}3$};
\node(4) at (7.7,0.5) {$\color{black}4$};\node(5) at (8.7,0.5) {$\color{black}5$};\node(6) at (9.7,0.5) {$\color{black}6$};

\end{tikzpicture}\end{center}
\caption{Correspondence between path-$3$-bipartite and star-$3$-bipartite}\label{crsp}\end{figure}
\end{exa}

\subsection{Definitions related to assignment problems}

Let $S_n$ denote the set of  permutations on $[n]$, and $S_{I,J}$ denote the set of  bijections from $I\subseteq[n]$ to $J\subseteq [n]$ (that is, $|I|=|J|$).  
Recall that the zero element~$\zero$ of $\R_{\max}$ is $-\infty$, 
the unit element~$\unit$ of $R_{\max}$ is $0$,  
and for $A\in\mathbb{R}_{\max}^{n\times n}$ the {max-plus permanent} is given by 
$$\per(A)=\max_{\pi\in S_n} \sum_{i\in [n]} A_{i, \pi(i)}.$$

We recall the correspondence between weighted simple digraphs and square matrices.
Let~$G=(V,E,w)$ be a weighted digraph, where~$w(e)$ denotes the weight of the edge $e\in E$.
For an~$n\times n$ matrix $M$, we associate a weighted simple digraph $G_M=([n], E,w)$, 
where for all $M_{i,j}\ne \zero$, $(i,j)\in E$ with~$w(i,j)=M_{i,j}$.
Conversely, for a weighted simple digraph~$G=([n],E,w)$ we associate an~$n\times n$ weight matrix~$M_G$,
where 
$$M_{i,j}=\begin{cases}
w{(i,j)}&;\  \text{if } (i,j)\in E,\\
\zero&;\  \text{otherwise}.
\end{cases}$$
 
See for instance the  correspondence in Figure~\ref{matdig}  between the $3\times 3$ matrix $M$ and the weighted simple digraph
$G$ with node-set $[3]$, where the edge-value denotes its weight. 

\begin{figure}[h]\begin{center}\begin{tikzpicture}
[main_node/.style={->,>=stealth',shorten >=1pt,circle,fill=black,minimum size=0.8em,inner sep=1pt,line width=0.3mm},
sx_node/.style={->,>=stealth',shorten >=1pt,circle,fill=white,minimum size=0.8em,inner sep=1pt,line width=0.3mm},
sec_node/.style={->,>=stealth',shorten >=1pt,circle,out=30,in=150,fill=white,minimum size=0.8em,inner sep=1pt,line width=0.3mm},
thr_node/.style={->,>=stealth',shorten >=1pt,circle,out=-150,in=-30,fill=black,minimum size=0.8em,inner sep=1pt,line width=0.3mm},
fr_node/.style={->,>=stealth',shorten >=1pt,circle,out=150,in=270,fill=black,minimum size=0.8em,inner sep=1pt,line width=0.3mm},
ft_node/.style={->,>=stealth',shorten >=1pt,circle,out=-30,in=90,fill=black,minimum size=0.8em,inner sep=1pt,line width=0.3mm},
prl_node/.style={-,>=stealth',shorten >=1pt,circle,out=240,in=120,fill=black,minimum size=0.8em,inner sep=1pt,line width=0.5mm},
prr_node/.style={-,>=stealth',shorten >=1pt,circle,out=-60,in=60,fill=black,minimum size=0.8em,inner sep=1pt,line width=0.5mm}]

\node[draw=black,circle,white](1) at (-1.3,3) {$\color{black}1$};\node[main_node](v1)at(-1,3){};
\node[draw=black,circle,white](2) at (3.3,3) {$\color{black}2$};\node[main_node](v2)at(3,3){};
\node[draw=black,circle,white](3) at (1,-0.4) {$\color{black}3$};\node[main_node](v3)at(1,0){};
\node[draw=black,circle,white](m11) at (-1,3.75) {$\color{black}M_{1,1}$};

\draw[main_node]  (v3) edge node{$M_{3,2}$} (v2);
 \draw[sx_node]  (v1) to [out=130,in=50,looseness=8] (v1);  
\draw[sec_node]   (v1) edge node{$M_{1,2}$} (v2);  
 \draw[thr_node] (v2) edge node{$M_{2,1}$} (v1);  

   \draw[fr_node]  (v3) edge node{$M_{3,1}$} (v1);   
\draw[main_node]  (v1) edge node{$M_{1,3}$} (v3);  

\node[draw=black,circle,white](11) at (-7.4,3) {$\color{black}M_{1,1}$};
\node[draw=black,circle,white](12) at (-6.2,3) {$\color{black}M_{1,2}$};
\node[draw=black,circle,white](13) at (-5,3) {$\color{black}M_{1,3}$};

\node[draw=black,circle,white](21) at (-7.4,2) {$\color{black}M_{2,1}$};
\node[draw=black,circle,white](22) at (-6.2,2) {$\color{black}\zero$};
\node[draw=black,circle,white](23) at (-5,2) {$\color{black}\zero$};

\node[draw=black,circle,white](31) at (-7.4,1) {$\color{black}M_{3,1}$};
\node[draw=black,circle,white](32) at (-6.2,1) {$\color{black}M_{3,2}$};
\node[draw=black,circle,white](33) at (-5,1) {$\color{black}\zero$};

\node[draw=black,circle,white](s1) at (-4.75,3.5) {$\color{black}$};
\node[draw=black,circle,white](s2) at (-7.75,3.5) {$\color{black}$};
\node[draw=black,circle,white](s3) at (-4.75,0.5) {$\color{black}$};
\node[draw=black,circle,white](s4) at (-7.75,0.5) {$\color{black}$};
\node[draw=black,circle,white](M) at (-9.5,2.5) {$\color{black}M=M_G$};
\node[draw=black,circle,white](=) at (-8.5,2.5) {$\color{black}=$};
\node[draw=black,circle,white](G) at (-1.25,0) {$\color{black}G=G_M$};

\draw[prr_node]   (s1) edge node{} (s3);
\draw[prl_node]   (s2) edge node{} (s4);

\node[draw=black,circle,white](a1) at (-2,2) {$\color{black}$};
\node[draw=black,circle,white](a2) at (-3.75,2) {$\color{black}$};

\draw[main_node]   (a1) edge node{} (a2);\draw[main_node]   (a2) edge node{} (a1);
\end{tikzpicture}\end{center}\caption{Associated square matrix and simple weighted digraph}\label{matdig}\end{figure}
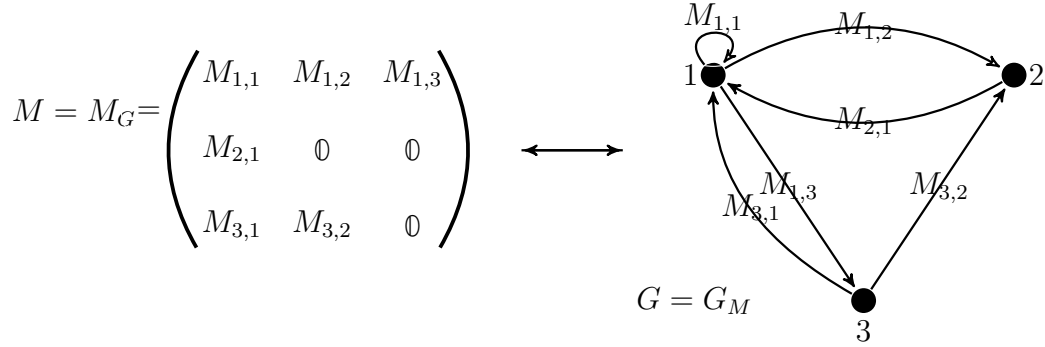

To a max-plus product~$P$ of entries of~$M\in \R^{n\times n}$, we  assign a sub-digraph with  the set  of edges~$E_P$
(multiplicities allowed) 
that correspond to the entries in the  max-plus  product, and the set of nodes~$V(E_P)$. 
In particular, we have the \textit{bijection-subdigraph}~$G=\big(V(E_{\pi}),E_{\pi}\big),$ where $\pi\in S_{I,J}$, corresponding to the max-plus product
 $P=\bigodot_{i\in [I]} M_{i, \pi(i)}$,
and satisfying 
$\deg^+(v)=\deg^-(u)=1,\ \forall v\in I, u\in J.$

Thus, in the sense of Lemma~\ref{eqv1},  $G$ corresponds  to a $1$-regular bipartite graph $B$, where $V_{B,1}=I,\ V_{B,2}=J$, 
 which is a perfect matching.
The set of permutations of maximal weight   in $M$ (or optimal permutations of $M$),  denoted by   
$$\ap(M)=\{\pi\in S_n : \per(M)=\bigodot_{i\in [n]} M_{i, \pi(i)}\},$$ 
 is identical to the set   of optimal solutions to the assignment problem in the graph 
corresponding to the digraph associated with $M$.

\begin{exa}\
\begin{enumerate}
\item Every 
non-$\zero$ (max-plus) summand $\bigodot_{i\in [n]}M_{i,\pi(i)}$ in the permanent of $M$ may be assigned 
with the permutation-subgraph of $G_M$ $$V(E_\pi)=[n],E_\pi=\{(i,\pi(i))\ \forall i\in[n]\}.$$  
\item The upper-right  bipartite subgraph and bold  perfect matching in Figure~\ref{crsp},   
correspond to  the  digraph and its  bold permutation in Figure~\ref{DPSD}

\begin{figure}[h]
\begin{center}\begin{tikzpicture}
[main_node/.style={->,>=stealth',shorten >=1pt,circle,fill=black,minimum size=0.8em,inner sep=1pt,line width=0.5mm},
sec_node/.style={->,>=stealth',shorten >=1pt,circle,fill=white,minimum size=0.8em,inner sep=1pt},
thr_node/.style={->,>=stealth',shorten >=1pt,circle,out=120,in=240,fill=black,minimum size=0.8em,inner sep=1pt,line width=0.5mm},
fr_node/.style={->,>=stealth',shorten >=1pt,circle,out=300,in=60,fill=black,minimum size=0.8em,inner sep=1pt,line width=0.5mm},
ft_node/.style={->,>=stealth',shorten >=1pt,circle,out=30,in=150,fill=black,minimum size=0.8em,inner sep=1pt,line width=0.5mm},
sx_node/.style={->,>=stealth',shorten >=1pt,circle,fill=white,minimum size=0.8em,inner sep=1pt,line width=0.5mm}]

\node[draw=black,circle,white](5) at (2.3,2) {$\color{black}5$};\node[main_node](v5)at(2,2){};
\node[draw=black,circle,white](3) at (2.3,1) {$\color{black}3$};\node[main_node](v3)at(2,1){};
\node[draw=black,circle,white](4) at (6.3,1) {$\color{black}4$};\node[main_node](v4)at(6,1){};
\node[draw=black,circle,white](6) at (-0.3,3) {$\color{black}6$};\node[main_node](v6)at(0,3){};
\node[draw=black,circle,white](1) at (3.7,3) {$\color{black}1$};\node[main_node](v1)at(4,3){};
\node[draw=black,circle,white](2) at (6.3,3) {$\color{black}2$};\node[main_node](v2)at(6,3){};

\draw[main_node]   (v1) edge node{} (v2);   \draw[main_node] (v2) edge node{} (v4);  \draw[main_node]  (v4) edge node{} (v1);
   \draw[thr_node]  (v3) edge node{} (v5);   \draw[fr_node]  (v5) edge node{} (v3);   \draw[sx_node]  (v6) to [out=130,in=50,looseness=8] (v6);  
 \draw[sec_node]  (v6) edge node{} (v5);
\end{tikzpicture}\end{center}\caption{ A digraph and  its permutation subgraph}\label{DPSD}\end{figure}
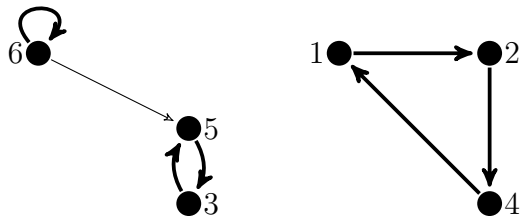

\end{enumerate}\end{exa}

\begin{df}We  say that a 
bipartite graph $H$ is  \textit{perfect} if $$\deg(v)=1\ \forall v\in V_{H,1}\cup V_{H,2}.$$
A disjoint-$k$-bipartite graph  $G=(V,E)$ is {perfect} if $\deg(v)=1$ for all $v\in V(G)$. 
Consequently, a path-$k$-bipartite or star-$k$-bipartite graph is called  {perfect} if it corresponds
 to a perfect disjoint-$k$-bipartite graph.

Formally, 
a~{path-$k$-bipartite graph} $H= (V_1,\dots, V_{k+1}, E)$ is {perfect} if 
 $$\deg^+(v)=\deg^-(u)=\deg^+(u)=\deg^-(w)=1,\ \forall\ v\in V_{H,1},\  u\in \bigcup_{\ell\in[k-2]}V_{H,\ell+1},\  w\in V_{H,k}\ ,$$ 
and a star-$k$-bipartite graph $G=(U_1,\dots, U_k, V, E_G)$ is {perfect} if for every $i\in[k]$
$$d(u)=1,\ \forall u\in U_i.$$
\end{df}

The following proposition is a result of  Lemma~\ref{eqv1}, Hall's Marriage Theorem~\cite{HMT} and Proposition~3.17,~\cite{STMA}.

\begin{pro}\label{eqv2}A $k$-regular digraph (and in particular multigraph)    corresponds to a perfect $k$-bipartite graph,
and its   edge-set is the disjoint union of edge-sets of $k$ permutation-subgraphs. \end{pro}
The correspondence   is demonstrated in Figure~\ref{kreg}.

\begin{figure}[h]\begin{center}\begin{tikzpicture}
[main_node/.style={-,>=stealth',shorten >=1pt,circle,fill=black,minimum size=0.8em,inner sep=1pt,line width=0.3mm},
maina_node/.style={->,>=stealth',shorten >=1pt,circle,out=130,in=50,looseness=8,fill=black,minimum size=0.8em,inner sep=1pt,line width=0.3mm},
mainr_node/.style={-,>=stealth',shorten >=1pt,circle,fill=black,minimum size=0.8em,inner sep=1pt,line width=0.3mm,color=red},
mainar_node/.style={->,>=stealth',shorten >=1pt,circle,fill=black,minimum size=0.8em,inner sep=1pt,line width=0.3mm,color=red},
mainb_node/.style={-,>=stealth',shorten >=1pt,circle,fill=black,minimum size=0.8em,inner sep=1pt,line width=0.3mm,color=blue},
mainabr_node/.style={->,>=stealth',shorten >=1pt,circle,out=30,in=150,fill=black,minimum size=0.8em,inner sep=1pt,line width=0.3mm,color=blue},
mainabl_node/.style={->,>=stealth',shorten >=1pt,circle,out=-130,in=310,fill=black,minimum size=0.8em,inner sep=1pt,line width=0.3mm,color=blue},
mainabloop_node/.style={->,>=stealth',shorten >=1pt,circle,out=310,in=230,looseness=8,fill=black,minimum size=0.8em,inner sep=1pt,line width=0.3mm,color=blue}]

\node[draw=black,circle,white](1) at (0.7,6) {$\color{black}1$};\node[main_node](l1)at(1,6){};
\node[draw=black,circle,white](2) at (2.3,6) {$\color{black}2$};\node[main_node](l2)at(2,6){};
\node[draw=black,circle,white](3) at (1.8,4) {$\color{black}3$};\node[main_node](l3)at(1.5,4){};

\node[draw=black,circle,white](cong1) at (3.5,5) {$\color{black}\cong$};

\node[draw=black,circle,white](11) at (4.7,7) {$\color{black}1$};\node[main_node](m1)at(5,7){};
\node[draw=black,circle,white](12) at (4.7,5) {$\color{black}2$};\node[main_node](m2)at(5,5){};
\node[draw=black,circle,white](13) at (4.7,3) {$\color{black}3$};\node[main_node](m3)at(5,3){};
\node[main_node](m4)at(6,7){};
\node[main_node](m5)at(6,5){};
\node[main_node](m6)at(6,3){};
\node[main_node](m7)at(7,7){};
\node[main_node](m8)at(7,5){};
\node[main_node](m9)at(7,3){};
\node[main_node](m10)at(8,7){};
\node[main_node](m11)at(8,5){};
\node[main_node](m12)at(8,3){};

\node[draw=black,circle,white](cong2) at (9,5) {$\color{black}\cong$};

\node[draw=black,circle,white](21) at (11.5,5.3) {$\color{black}1$};\node[main_node](r1)at(11.5,5){};
\node[draw=black,circle,white](22) at (12.5,5.3) {$\color{black}2$};\node[main_node](r2)at(12.5,5){};
\node[draw=black,circle,white](23) at (13.5,5.3) {$\color{black}3$};\node[main_node](r3)at(13.5,5){};
\node[main_node](r4)at(10,7){};
\node[main_node](r5)at(11,7){};
\node[main_node](r6)at(12,7){};
\node[main_node](r7)at(13,7){};
\node[main_node](r8)at(14,7){};
\node[main_node](r9)at(15,7){};
\node[main_node](r10)at(11.5,3){};
\node[main_node](r11)at(12.5,3){};
\node[main_node](r12)at(13.5,3){};

\draw[mainr_node](r4)edge node{}(r2);
\draw[mainr_node](r5)edge node{}(r3);
\draw[mainr_node](r6)edge node{}(r1);
\draw[mainr_node](m1)edge node{}(m5);
\draw[mainr_node](m2)edge node{}(m6);
\draw[mainr_node](m3)edge node{}(m4);
\draw[mainar_node](l1)edge node{}(l2);
\draw[mainar_node](l2)edge node{}(l3);
\draw[mainar_node](l3)edge node{}(l1);

\draw[mainb_node](r7)edge node{}(r2);
\draw[mainb_node](r9)edge node{}(r3);
\draw[mainb_node](r8)edge node{}(r1);
\draw[mainb_node](m7)edge node{}(m5);
\draw[mainb_node](m9)edge node{}(m6);
\draw[mainb_node](m8)edge node{}(m4);
\draw[mainabr_node](l1)edge node{}(l2);
\draw[mainabl_node](l2)edge node{}(l1);
\draw[mainabloop_node](l3)edge node{}(l3);

\draw[main_node](r12)edge node{}(r3);
\draw[main_node](r11)edge node{}(r2);
\draw[main_node](r10)edge node{}(r1);
\draw[main_node](m7)edge node{}(m10);
\draw[main_node](m9)edge node{}(m12);
\draw[main_node](m8)edge node{}(m11);
\draw[maina_node](l1)edge node{}(l1);
\draw[maina_node](l2)edge node{}(l2);
\draw[maina_node](l3)edge node{}(l3);

\end{tikzpicture}\end{center}\caption{Correspondence between $3$-regular digraph and perfect $3$-bipartites.}\label{kreg}\end{figure}
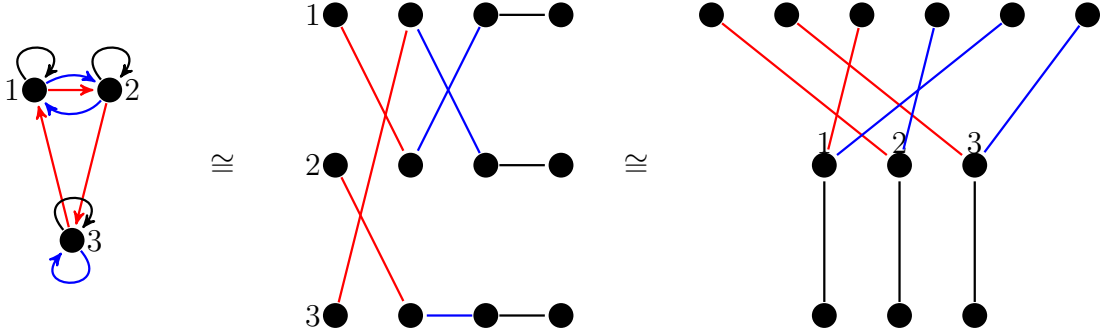

\begin{df}\label{kl} Let $G=([n],E)$ be a~$ {k}$-regular multigraph, and let $$E=\biguplus_{i\in[k]}E_{\rho_i},\ \rho_i\in S_n$$ 
be the disjoint union of edge-sets of $k$ permutation-subgraphs $G_i=([n],E_{\rho_i})$ of $G$.
We say $G$ is~\textit{$(\ell,k)$-regular} with respect to $\{I_j,J_j: |I_j|=|J_j|=k\ \forall  j\in[\ell]\}$ if there exist $e_{i_1},...,e_{i_\ell}\in E_{\rho_i}\ \ \forall i\in[k]$ 
such that$$\big(V(E_{\pi_j}),E_{\pi_j}=\{e_{1_j},..
.,e_{k_j}\}\big),\ j\in[\ell],\ \pi_j\in S_{I_j,J_j}$$ are  $\ell$ disjoint  bijection-subgraphs.   (In particular 
$|E_{\pi_j}|=k,\ \forall j
\in[\ell]$.) We denote $$G=\big([n],\biguplus_{i\in[k]}E_{\rho_i},\{\pi_j\}_{j\in\ell}\big).$$
% We say two $ {k}$-regular multigraphs $G_1=\big([n],E_1=\biguplus_{i\in[k]}E_{\rho_i}\big)$ and $G_2=\big([n],E_2=\biguplus_{i\in[k]}E_{\sigma_i}\big)$ are equal if $\rho_i=\sigma_i\ \forall i$.
\end{df}

\begin{exa}\label{asgnmt}\
\begin{enumerate}
\item The digraph in Figure~\ref{kreg} is $(1,3)$-regular with respect to $I=J=[3]$, where $\pi=(1\ 2)$ is constructed from red edge $(1,2)$,
blue edge $(2,1)$ and black loop $(3,3)$, 
as indicated  by green edges in
 Figure~\ref{kreg'}. This is a special case where $n=k$.

\begin{center}\begin{tikzpicture}
[main_node/.style={-,>=stealth',shorten >=1pt,circle,fill=black,minimum size=0.8em,inner sep=1pt,line width=0.5mm},
maina_node/.style={->,>=stealth',shorten >=1pt,circle,out=130,in=50,looseness=8,fill=black,minimum size=0.8em,inner sep=1pt,line width=0.3mm},
mainr_node/.style={-,>=stealth',shorten >=1pt,circle,fill=black,minimum size=0.8em,inner sep=1pt,line width=0.3mm,color=red},
mainar_node/.style={->,>=stealth',shorten >=1pt,circle,fill=black,minimum size=0.8em,inner sep=1pt,line width=0.3mm,color=red},
mainb_node/.style={-,>=stealth',shorten >=1pt,circle,fill=black,minimum size=0.8em,inner sep=1pt,line width=0.3mm,color=blue},
mainabr_node/.style={->,>=stealth',shorten >=1pt,circle,out=30,in=150,fill=black,minimum size=0.8em,inner sep=1pt,line width=0.3mm,color=blue},
mainabl_node/.style={->,>=stealth',shorten >=1pt,circle,out=-130,in=310,fill=black,minimum size=0.8em,inner sep=1pt,line width=0.3mm,color=blue},
mainadl_node/.style={-,>=stealth',shorten >=1pt,circle,out=-30,in=100,fill=black,minimum size=0.8em,inner sep=1pt,line width=0.3mm,color=blue},
mainafl_node/.style={-,>=stealth',shorten >=1pt,circle,out=-40,in=210,fill=black,minimum size=0.8em,inner sep=1pt,line width=0.3mm,color=blue},
mainabloop_node/.style={->,>=stealth',shorten >=1pt,circle,out=310,in=230,looseness=8,fill=black,minimum size=0.8em,inner sep=1pt,line width=0.3mm,color=blue}]

\node[draw=black,circle,white](1) at (0.7,6) {$\color{black}1$};\node[main_node](l1)at(1,6){};
\node[draw=black,circle,white](2) at (2.3,6) {$\color{black}2$};\node[main_node](l2)at(2,6){};
\node[draw=black,circle,white](3) at (1.8,4) {$\color{black}3$};\node[main_node](l3)at(1.5,4){};

\node[draw=black,circle,white](cong1) at (3.5,5) {$\color{black}\cong$};

\node[draw=black,circle,white](11) at (4.7,7) {$\color{black}1$};\node[main_node](m1)at(5,7){};
\node[draw=black,circle,white](12) at (4.7,5) {$\color{black}2$};\node[main_node](m2)at(5,5){};
\node[draw=black,circle,white](13) at (4.7,3) {$\color{black}3$};\node[main_node](m3)at(5,3){};
\node[main_node](m4)at(6,7){};
\node[main_node](m5)at(6,5){};
\node[main_node](m6)at(6,3){};

\draw[mainadl_node,color=green](m1)edge node{}(m5);
\draw[mainr_node](m2)edge node{}(m6);
\draw[mainr_node](m3)edge node{}(m4);
\draw[mainar_node,color=green](l1)edge node{}(l2);
\draw[mainar_node](l2)edge node{}(l3);
\draw[mainar_node](l3)edge node{}(l1);

\draw[mainr_node,color=green](m2)edge node{}(m4);
\draw[main_node,color=green](m3)edge node{}(m6);

\draw[main_node,color=black](m1)edge node{}(m4);
\draw[main_node,color=black](m2)edge node{}(m5);

\draw[mainr_node,color=blue](m1)edge node{}(m5);
\draw[mainafl_node,color=blue](m3)edge node{}(m6);

\draw[mainabr_node](l1)edge node{}(l2);
\draw[mainabl_node,color=green](l2)edge node{}(l1);
\draw[mainabloop_node](l3)edge node{}(l3);

\draw[maina_node](l1)edge node{}(l1);
\draw[maina_node](l2)edge node{}(l2);
\draw[maina_node,color=green](l3)edge node{}(l3);

\node[draw=black,circle,white](cong2) at (7,5) {$\color{black}\cong$};

\node[draw=black,circle,white](1) at (7.7,7) {$\color{black}1$};\node[main_node](l1)at(1,6){};
\node[draw=black,circle,white](2) at (7.7,5) {$\color{black}2$};\node[main_node](l2)at(2,6){};
\node[draw=black,circle,white](3) at (7.7,3) {$\color{black}3$};\node[main_node](l3)at(1.5,4){};

\node[main_node](m10)at(8,7){};
\node[main_node](m11)at(8,5){};
\node[main_node](m12)at(8,3){};
\node[main_node](m13)at(9,7){};
\node[main_node](m14)at(9,5){};
\node[main_node](m15)at(9,3){};

\node[draw=black,circle,white](1) at (9.7,7) {$\color{black}1$};\node[main_node](l1)at(1,6){};
\node[draw=black,circle,white](2) at (9.7,5) {$\color{black}2$};\node[main_node](l2)at(2,6){};
\node[draw=black,circle,white](3) at (9.7,3) {$\color{black}3$};\node[main_node](l3)at(1.5,4){};

\node[main_node](m16)at(10,7){};
\node[main_node](m17)at(10,5){};
\node[main_node](m18)at(10,3){};
\node[main_node](m19)at(11,7){};
\node[main_node](m20)at(11,5){};
\node[main_node](m21)at(11,3){};

\node[draw=black,circle,white](1) at (11.7,7) {$\color{black}1$};\node[main_node](l1)at(1,6){};
\node[draw=black,circle,white](2) at (11.7,5) {$\color{black}2$};\node[main_node](l2)at(2,6){};
\node[draw=black,circle,white](3) at (11.7,3) {$\color{black}3$};\node[main_node](l3)at(1.5,4){};

\node[main_node](m22)at(12,7){};
\node[main_node](m23)at(12,5){};
\node[main_node](m24)at(12,3){};
\node[main_node](m7)at(13,7){};
\node[main_node](m8)at(13,5){};
\node[main_node](m9)at(13,3){};

\draw[mainr_node,color=green](m10)edge node{}(m14);
\draw[mainr_node,color=green](m17)edge node{}(m19);
\draw[main_node,color=green](m24)edge node{}(m9);

\draw[mainr_node,color=red](m11)edge node{}(m15);
\draw[mainr_node,color=red](m12)edge node{}(m13);

\draw[mainr_node,color=blue](m16)edge node{}(m20);
\draw[main_node,color=blue](m18)edge node{}(m21);

\draw[main_node,color=black](m22)edge node{}(m7);
\draw[main_node,color=black](m23)edge node{}(m8);

\end{tikzpicture}\end{center}

%%%%%%%%%%%%%%%%%%%%%%%%%%%%%%%%%%%%%%%%

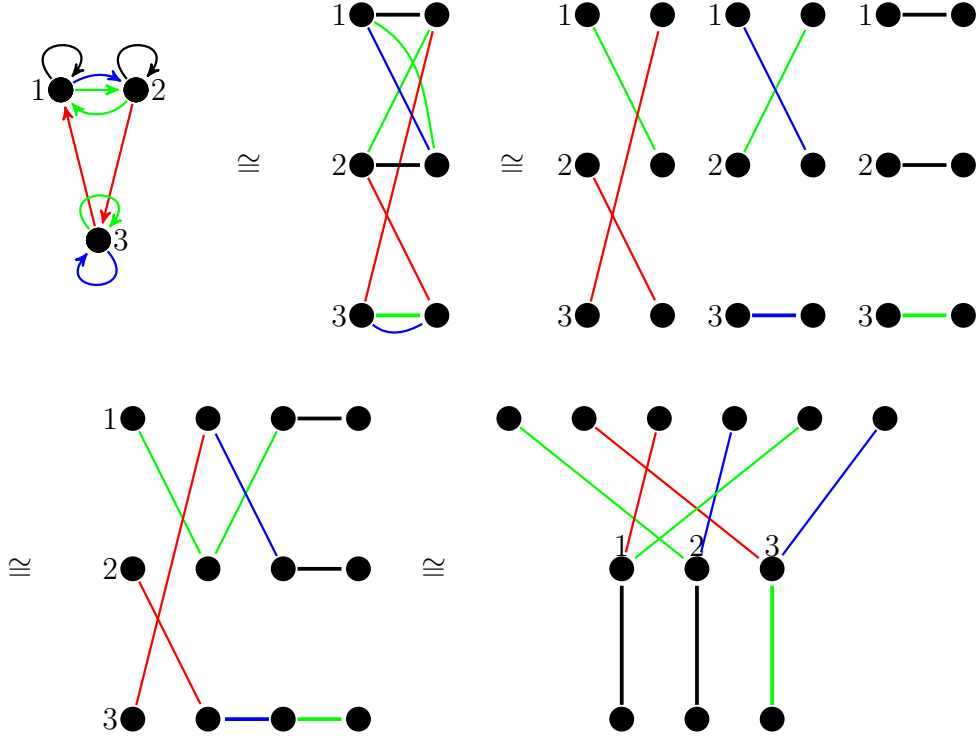
\begin{figure}[h]\begin{center}\begin{tikzpicture}
[main_node/.style={-,>=stealth',shorten >=1pt,circle,fill=black,minimum size=0.8em,inner sep=1pt,line width=0.5mm},
maina_node/.style={->,>=stealth',shorten >=1pt,circle,out=130,in=50,looseness=8,fill=black,minimum size=0.8em,inner sep=1pt,line width=0.3mm},
mainr_node/.style={-,>=stealth',shorten >=1pt,circle,fill=black,minimum size=0.8em,inner sep=1pt,line width=0.3mm,color=red},
mainar_node/.style={->,>=stealth',shorten >=1pt,circle,fill=black,minimum size=0.8em,inner sep=1pt,line width=0.3mm,color=red},
mainb_node/.style={-,>=stealth',shorten >=1pt,circle,fill=black,minimum size=0.8em,inner sep=1pt,line width=0.3mm,color=blue},
mainabr_node/.style={->,>=stealth',shorten >=1pt,circle,out=30,in=150,fill=black,minimum size=0.8em,inner sep=1pt,line width=0.3mm,color=blue},
mainabl_node/.style={->,>=stealth',shorten >=1pt,circle,out=-130,in=310,fill=black,minimum size=0.8em,inner sep=1pt,line width=0.3mm,color=blue},
mainabloop_node/.style={->,>=stealth',shorten >=1pt,circle,out=310,in=230,looseness=8,fill=black,minimum size=0.8em,inner sep=1pt,line width=0.3mm,color=blue}]

\node[draw=black,circle,white](cong1) at (3.5,5) {$\color{black}\cong$};

\node[draw=black,circle,white](11) at (4.7,7) {$\color{black}1$};\node[main_node](m1)at(5,7){};
\node[draw=black,circle,white](12) at (4.7,5) {$\color{black}2$};\node[main_node](m2)at(5,5){};
\node[draw=black,circle,white](13) at (4.7,3) {$\color{black}3$};\node[main_node](m3)at(5,3){};
\node[main_node](m4)at(6,7){};
\node[main_node](m5)at(6,5){};
\node[main_node](m6)at(6,3){};
\node[main_node](m7)at(7,7){};
\node[main_node](m8)at(7,5){};
\node[main_node](m9)at(7,3){};
\node[main_node](m10)at(8,7){};
\node[main_node](m11)at(8,5){};
\node[main_node](m12)at(8,3){};

\node[draw=black,circle,white](cong2) at (9,5) {$\color{black}\cong$};

\node[draw=black,circle,white](21) at (11.5,5.3) {$\color{black}1$};\node[main_node](r1)at(11.5,5){};
\node[draw=black,circle,white](22) at (12.5,5.3) {$\color{black}2$};\node[main_node](r2)at(12.5,5){};
\node[draw=black,circle,white](23) at (13.5,5.3) {$\color{black}3$};\node[main_node](r3)at(13.5,5){};
\node[main_node](r4)at(10,7){};
\node[main_node](r5)at(11,7){};
\node[main_node](r6)at(12,7){};
\node[main_node](r7)at(13,7){};
\node[main_node](r8)at(14,7){};
\node[main_node](r9)at(15,7){};
\node[main_node](r10)at(11.5,3){};
\node[main_node](r11)at(12.5,3){};
\node[main_node](r12)at(13.5,3){};

\draw[mainr_node,color=green](r4)edge node{}(r2);
\draw[mainr_node](r5)edge node{}(r3);
\draw[mainr_node](r6)edge node{}(r1);
\draw[mainr_node,color=green](m1)edge node{}(m5);
\draw[mainr_node](m2)edge node{}(m6);
\draw[mainr_node](m3)edge node{}(m4);

\draw[mainb_node](r7)edge node{}(r2);
\draw[mainb_node](r9)edge node{}(r3);
\draw[mainb_node,color=green](r8)edge node{}(r1);
\draw[mainb_node,color=green](m7)edge node{}(m5);
\draw[main_node,color=blue](m9)edge node{}(m6);
\draw[mainb_node](m8)edge node{}(m4);

\draw[main_node,color=green](r12)edge node{}(r3);
\draw[main_node](r11)edge node{}(r2);
\draw[main_node](r10)edge node{}(r1);
\draw[main_node](m7)edge node{}(m10);
\draw[main_node,color=green](m9)edge node{}(m12);
\draw[main_node](m8)edge node{}(m11);

\end{tikzpicture}\end{center}\caption{$(1,3)$-regular.}\label{kreg'}\end{figure}

\item If~$G=\big([n],\biguplus_{i\in[k]}E_{\rho}\big)$ for some $ \rho\in S_n$,
then for every~$I\subseteq[n]:\ |I|=k$,  the digraph~$G$ 
is~$(k,k)$-regular with respect to~$I=I_j,J=J_j,\ \forall j\in[k]$,  
where$$\pi_j=\pi=\rho|_I\in S_{I,J},\ \forall j\in[k].$$

It is also $(n-k+1,k)$-regular with respect to 
$$I_j=\{t+j-1:\ t\in[k]\}\ ,\ J_j=\{\rho (t+j-1):\ t\in[k]\},\ \ j\in[n-k+1],$$ 
where $$E_{\pi_{j+1}}=\big\{(j+k,\rho(j+k))\big\}\biguplus  E_{\pi_j}\setminus\big\{(j,\rho(j))\big\}.$$ 
\item A filled Sudoku table represents an $(\ell,k)$-regular digraph, with respect to $I=J=[9]$, where $n=k=\ell=9$. 
\end{enumerate}\end{exa}

The rough idea behind Example~\ref{asgnmt} part~(1) is to present a set of assignments of people to jobs, using permutations over multi-digraphs.  
Observing Figure~\ref{kreg'}, this can be applied  in different equivalent ways. For instance, assigning different sets of jobs to the same team, 
or the same set of jobs to different teams, as described by the lower-right star-$3$-bipartite graph. One may assign a team 
of experienced workers to tutor new workers, which 6 months later will tutor newer workers, and so on, 
 as described by the lower-left path-$3$-bipartite graph. In this paper, we consider assigning the same set of jobs, 
to the same team over a given  time interval, as described in the middle-upper  multi-digraph, 
or assigning different jobs to different teams, as described in the right-upper   disjoint-$3$-bipartite graph.

Additionally, we  want to be able to consider assignments    performed under some restrictions. 
That is, for a single assignment relating to $G_A$, one edge is fixed, and we consider all 
 perfect matchings including this edge.

\begin{df} Let~$A\in\R_{\max}^{n\times m}$. We denote by  $A^{\wedge k}\in\R_{\max}^{\left(\substack{n\\k}\right)\times \left(\substack{m\\k}\right)} $  the max-plus \textit{$k^{th}$ 
compound matrix of $A$} defined by 
$$A^{\wedge k}_{I,J}=\max_{\substack{\sigma:I\rightarrow J\\\text{bijection}}}
\sum_{i\in I}A_{i,\sigma(i)}\ \ \forall I\subseteq[n]\text{ and }J\subseteq[m]:\ |I|=|J|=k.$$ 
In particular, $A^{\wedge 1}=A,\ A^{\wedge 0}=\unit$ and  $\per(A)=A^{\wedge n} $ is the {max-plus permanent} 
of~$A$ when $n=m$. In this case we denote by $\adj(A)_{i,j}=A^{\wedge n-1}_{\{j\}^c,\{i\}^c}$   the $(i\ j)$ entry of the  \textit{max-plus adjoint} of $A$.\end{df}

For a non-$\zero$ product of  concatenating entries of $M\in \R^{n\times n}$ 
$$P=M_{j_1,j_2}\odot M_{j_2,j_3}\odot\cdots\odot M_{j_{k},j_{k+1}},$$
the edge-set of the digraph $G=\big(V(E_P),E_P\big)$ corresponds to a {path} $$p=e_1\cdots e_k:\ s(e_i)=j_i,t(e_i)=j_{i+1}\ \forall i\in[k]$$ of length $l(p)=k$,  
from $s(p)=s(e_1)=j_1\in V(E_P)$ to $t(p)=t(e_k)=j_{k+1}\in V(E_P)$, 
and with  weight $w(p)=\bigodot_{i\in[k]}w(e_i,G).$ When it is clear which matrix we are using we may  just write $w(p)$.

We decompose   a bijection $\rho\in S_{I,J}$ (and in particular, when $I=J$, a permutation)
 into disjoint cycles whose set is denoted by $\mathcal{C}$, and elementary paths whose set is denoted by $\mathcal{P}$. This corresponds to the restrictions   $\rho|_{\overline{s(p)}},\ s(p)\in I$, where $$\overline{s(p)}=\{j\in I:\text{ s.t.~}
\rho^m(s(p))=j\text{ for some }m\in\mathbb{N}\}.$$  The quotient set  $$\big\{\overline{s(p)}:p\text{ is a cycle or an elementary path of }\rho\big\}=$$
$$\underbrace{\big\{ \overline{s(p)}:\rho^m(s(p))\in I,\forall m\in\mathbb{N}\big\}}_{\text{corresponds to }\mathcal{C}}
\cup
\underbrace{\big\{ \overline{s(p)}:\exists m_p\in\mathbb{N}\text{ s.t.~}\rho^{m_p}(s(p))\not\in I\big\}}_{\text{corresponds to }\mathcal{P}}$$ 
is the partition of $I$ corresponding to the  source-nodes of the disjoint cycles and paths of $\rho$.
Moreover, a  bijection-subgraph $G'=\big(V(E_\rho),E_\rho\big)\subseteq G=([n],E)$ satisfies 
$$E_\rho=\{e\in E:\ t(e)=\rho(s(e))\ \forall s(e)\in I\}\ ,\ \deg^+(i)=\deg^-(\rho(i))=1\ \forall \rho(i)\in J,i\in I,$$ and 
$$w(\rho)=\bigodot_{\overline{s(p)}}\bigodot_{s(e)\in\overline{s(p)}}w(e)=
\bigodot_{\overline{s(p)}\in \mathcal{C}}w(p)\bigodot_{\overline{s(p)}\in \mathcal{P}}w(p),\text{ where }\rho^{m_p}(s(p))=\begin{cases}s(p),\overline{s(p)}\in \mathcal{C}\\
t(p),\overline{s(p)}\in \mathcal{P}\enspace .\end{cases}$$
(That is, over $\rmax$, the weight of a permutation is the sum of weights of its disjoint cycles and paths, which is the sum 
of weights of its edges).

In particular, every elementary path (resp.~cycle) is  a bijection (resp.~permutation) from its set of
source-nodes to its set of target-nodes, and can be extended to a bijection $J^c\cup I\longrightarrow 
I^c\cup J$ (resp.~permutation in $S_n$) by 
composing it with the loops $$e:\ s(e)=t(e)=j,\ \forall j\in I^c\cap J^c\text{ (resp.~}j\in I^c).$$
Every bijection $\rho\in S_{I,J}$ can be extended to a permutation $\pi\in S_{I\cup J}$ by defining $$\pi(j)=\begin{cases}\rho(j)&,\  j\in I\\s(p)&,\  j=t(p)\in J\setminus I\end{cases}.$$
In the same way, every permutation~$\pi\in S_{I}=S_{I,I}$  can be reduced to a   
bijection $\rho\in S_{J,K},$ $J,K\subseteq I:\ K=\{\pi(j):j\in J\}$  by 
defining~$\rho=\pi|_J$. 

We note that a path $p$ from $s(p)$ to $t(p)$ can be decomposed  into (not necessarily disjoint) cycles 
and an elementary path from $s(p)$ to $t(p)$.

\begin{df}\label{optimal} Let  $I,J\subseteq [n]$ such that 
$|I|=|J|=k$. We say that $\pi\in S_n$ is 
an \textit{optimal permutation}   in a simple digraph $G=([n],E)$ weighted over $\R_{\max}$, if $w(\pi)\geq w(\tau)\ \forall \tau\in S_n$, or equivalently $$\per(M_G)=\sum_{i\in[n]}(M_G)_{i,\pi(i)}.$$  
We say that $\sigma\in S_{I,J}$ is  an \textit{optimal bijection with respect to $I,J$}  in $G$   
if $w(\sigma)\geq w(\rho)\ \forall \rho\in S_{I,J}$, or equivalently 
$$(M_G)^{\wedge k}_{_{I,J}}=
\sum_{i\in I}(M_G)_{i,\sigma(i)}.$$
We say that~$\big([n],\biguplus_{i\in[k]}E_{\rho_i},\sigma\big)$ is an \textit{optimal $(1,k)$-regular multigraph} of~$G$ 
 with respect to $I,J$
if 
$$\bigg(\sum_{i\in[k]}w(\rho_i)\bigg)-w( \sigma)\geq
\bigg(\sum_{i\in[k]}w(\rho'_i)\bigg)-w( \sigma'),$$ for every 
$(1,k)$-regular multigraph  $\big([n],\biguplus_{i\in[k]}E_{\rho'_i},\sigma'\big)$ of~$G$  with respect to $I,J$. Equivalently
$$\big(\adj(M_G)\big)^{\wedge k}_{_{J,I}}=
\sum_{i\in I}\big(\adj(M_G)\big)_{\sigma(i),i}\ ,\text{ where }\big(\adj(M_G)\big)_{\sigma(i),i}=\sum_{j\in \{i\}^c}(M_G)_{j,\rho_i(j)}.$$
\end{df}

Note that, the rough idea here is to represent a set of assignments of people to jobs.  In the assignment problem for example, we have a matrix $A$ associated with a digraph $G_A=([n],E_G)$ on which we find an optimal permutation $\pi$, which can equivalently be viewed as a perfect matcing in a bipartite graph $B=([n],[n],E_B)$ corresponds 
to~$G_A$.  We want to extend this to the problem of multiple assignments.  Trivially $k$ assignments can be represented as $M_1+M_2+\dots+ M_k$ where each $M_r,\ r\in[k]$ is a perfect matching.  These can also be viewed as a perfect $k$-bipartite graph, or specific versions of a perfect star/path/disjoint-$k$-bipartite graph. 

More importantly, we  want to be able to consider assignments    performed under some restrictions. 
That is, for a single assignment relating to $G_A$, one edge is fixed (in some to-be-defined way), and we consider all 
 perfect matchings including this edge.

%%%%%%%%%%%%%%%%%%%%%%%%%%%%%%%%%%%%%%%%%%%%%%%%%%%%%

\section{Combinatorial interpretation of the compound matrix of the tropical adjoint}

We describe a new form of the assignment problem whose solutions are given by entries of the compound matrix of the tropical adjoint.

Given is a set of $n$ workers and $n$ assignments.  The entries of a matrix $M\in\rmax^{n\times n}$ represent the \emph{value} of person $i$ performing assignment $j$.  This value could, for example, be the negative cost of the assignment, the negative time taken, the persons experience level of performing the job, or some function of multiple factors.

The usual assignment problem finds an assignment of people to jobs which maximises the value.  The assignment problem can be solved in $\mathcal{O}(n^3)$ \cite{Munkres57} time by the Hungarian Algorithm \cite{Kuhn55} which transforms to a non-positive matrix $B=D_1\otimes M\otimes D_2$, $D_1, D_2$ diagonal.  This transformation preserves the set of permutations of maximum weight and $B$ additionally has the property that every permutation $\pi$ of maximum weight satisfies $B_{i,\pi(i)}=0, \forall i$.

Suppose that, within a workplace, the same $n$ jobs have to be performed each day, and one person is required to perform each task (possibly using specialist equipment).  Additionally, a supervisor wishes to observe, or train, a subset of his employees on various tasks/pieces of equipment throughout the week.  On each day he will observe one worker and one task.

Then, a set of $k$ assignments of workers to tasks is required which additionally has the property that, in each assignment, a different worker $i\in I$ performs a different task $j\in J$.

\begin{df}We define \emph{$k$ assignments with supervisions $I$ on $J$} to be permutations $\pi_1,\dots,\pi_k\in S_n$ with a supervision $\tau:I\rightarrow J$ with $\tau(i_t):=j_t, \forall t\in[k]$ with the property that $\pi_t(i_t)=j_t, \forall t\in[k].$  \end{df}

\begin{rem} Obviously, for a week of assignments where there is one supervision per day we set $k=5$.  We can generalise from \emph{weekly} to daily/hourly and so on by changing the value of $k$.
\end{rem}

\begin{exa}  Let $I=\{1,3,6\}$ and $J=\{2,3,5\}$.  Shown in Figure \ref{Fig:AssignmentswithSupervisions} is a set of three assignments $\pi_1,\dots, \pi_3\in S_6$ with the property that, distributed over the assignments is a permutation $\tau\in S_3$ representing the supervisions.

\begin{figure}[h]\begin {tikzpicture}[node distance =1cm and 2cm ,on grid ,semithick, Vertex/.style ={circle, draw, inner sep=0pt, minimum size=0.5cm}, Vertex*/.style ={rectangle, draw, fill=red!20, inner sep=0pt, minimum size=0.5cm}]

\node[Vertex*] (A1){1};\node[Vertex] (A2) [below =of A1] {2};\node[Vertex*] (A3) [below =of A2] {3};\node[Vertex] (A4)[below =of A3] {4};\node[Vertex] (A5) [below =of A4] {5};\node[Vertex*] (A6) [below =of A5] {6};

\node[Vertex] (B1)[right =of A1] {1};\node[Vertex*] (B2) [below =of B1] {2};\node[Vertex*] (B3) [below =of B2] {3};\node[Vertex] (B4)[below =of B3] {4};\node[Vertex*] (B5) [below =of B4] {5};\node[Vertex] (B6) [below =of B5] {6};

\node[Vertex*] (C1)[right =2cm of B1] {1};\node[Vertex] (C2) [below =of C1] {2};\node[Vertex*] (C3) [below =of C2] {3};\node[Vertex] (C4)[below =of C3] {4};\node[Vertex] (C5) [below =of C4] {5};\node[Vertex*] (C6) [below =of C5] {6};

\node[Vertex] (D1)[right =of C1] {1};\node[Vertex*] (D2) [below =of D1] {2};\node[Vertex*] (D3) [below =of D2] {3};\node[Vertex] (D4)[below =of D3] {4};\node[Vertex*] (D5) [below =of D4] {5};\node[Vertex] (D6) [below =of D5] {6};

\node[Vertex*] (E1)[right =2cm of D1] {1};\node[Vertex] (E2) [below =of E1] {2};\node[Vertex*] (E3) [below =of E2] {3};\node[Vertex] (E4)[below =of E3] {4};\node[Vertex] (E5) [below =of E4] {5};\node[Vertex*] (E6) [below =of E5] {6};

\node[Vertex] (F1)[right =of E1] {1};\node[Vertex*] (F2) [below =of F1] {2};\node[Vertex*] (F3) [below =of F2] {3};\node[Vertex] (F4)[below =of F3] {4};\node[Vertex*] (F5) [below =of F4] {5};\node[Vertex] (F6) [below =of F5] {6};

\draw[thick] (A1) to (B3);\draw[thick] (A2) to (B1);
\draw[thick] (A4) to (B4);\draw[thick] (A5) to (B6);\draw[thick] (A6) to (B2);
\draw[thick] (C2) to (D2);\draw[thick] (C3) to (D4);
\draw[thick] (C4) to (D1);\draw[thick] (C5) to (D6);\draw[thick] (C6) to (D5);
\draw[thick] (E2) to (F1);\draw[thick] (E3) to (F3);
\draw[thick] (E4) to (F5);\draw[thick] (E5) to (F6);\draw[thick] (E1) to (F4);

\draw[very thick, dotted, color=red] (A3) to (B5);\draw[very thick, dotted, color=red] (C1) to (D3);\draw[very thick, dotted, color=red] (E6) to (F2);
\end{tikzpicture}\caption{3 Assignments with supervisions $\{1,3,6\}$ on $\{2,3,5 \}$.}\label{Fig:AssignmentswithSupervisions}\end{figure}
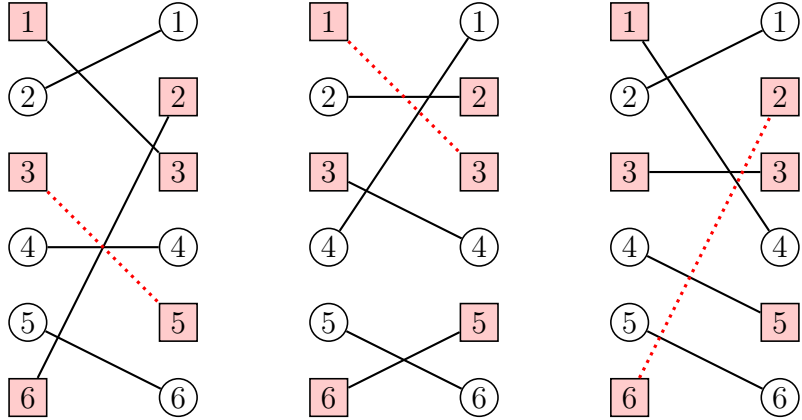
\end{exa}

Given that the supervisor has deemed it necessary that each worker $i\in I$ performs one of the tasks $j\in J$, it is accepted that the full assignment of people to jobs may not be optimal under this condition (for example if the supervisor wishes to train a worker $i$ for job $j$, the worker won't have experience of this job, and the value $M_{i,j}$ will be low).  Thus, we aim to optimise the value of all other assignments throughout the week, ignoring the value of the supervisions $I$ on $J$.  This leads to the following definition.

\begin{df} Let $M\in\rmax^{n\times n}$. The \emph{base value} of $k$ assignments $\rho_t$ for $t\in[k]$ with 
supervisions $(i_t,j_t)\in I\times J$ defined by a bijection $\sigma\colon I\to J$ 
is the weight of the $(1,k)$-regular multigraph $\big([n],\biguplus_{t\in[k]}E_{\rho_t},\sigma\big)$  
$$\sum_{t\in[k]} \left(w(\rho_t, M)-M_{i_t,\sigma(i_t)}\right)=\sum_{t=1}^k \sum_{i\neq i_t}M_{i,\rho_t(i)}.$$
\end{df}

\begin{exa} The base value for the 3 assignments with supervisions shown in Figure~\ref{Fig:AssignmentswithSupervisions} is \begin{align*}&(M_{1,3}+M_{2,1}+M_{4,4}+M_{5,6}+M_{6,2})\\+&(M_{2,2}+M_{3,4}+M_{4,1}+M_{5,6}+M_{6,5})\\+&(M_{1,4}+M_{2,1}+M_{3,3}+M_{4,5}+M_{5,6}).\end{align*}
\end{exa}

\begin{obs} The optimal (with respect to $M$) base value of $k$ assignments with supervisions $I$ on $J$ is 
equal to the weight of optimal $(1,k)$-regular multigraph with respect to $I$ and $J$, which is
$$\bigoplus_{\sigma\in S_{J,I}}w(\sigma, \adj(M)_{J,I})=[\adj(M)]^{\wedge k}_{J,I}.$$
\end{obs}

\begin{nota}
Solution of an optimization problem is in general non-unique. Below we will indicate multiple optimal solutions by the superscript $\bullet$, inspired  
by~\cite{BCOQ92,G.Ths}. 
\end{nota}

\begin{nota}
For a bijection $\beta: I\rightarrow J$ we will also use the notation $$\beta=(i_1, j_1), (i_2, j_2),\dots, (i_k, j_k)$$ where $\beta(i_t)=j_t$ for $t\in[k]$.
\end{nota}

\begin{exa}\label{Cexa} Let $$M=\begin{pmatrix}0&1&-2&-4\\-3&0&5&2\\-5&4&0&6\\-1&-6&3&0
\end{pmatrix},\text{ then }\adj(M)=\begin{pmatrix}9&10&6^{\bullet}&12\\ 10&9&5^{\bullet}&11\\ 5&6&2&6^{\bullet}\\ 8&9&5&9 \end{pmatrix}.$$
Suppose we want 2 assignments of people to jobs with maximum base value such that we additionally supervise $I=\{2,4\}$ on $J=\{1,2\}$.

The maximum base value is given by $$\adj(M)^{\wedge 2}_{J, I}=\per\begin{pmatrix}10&12\\9&11\end{pmatrix}=21^{\bullet}.$$  
Here, after accounting for the transposition in $\adj(M)$, we get that the bijections/su-pervisions attaining this maximum value are $\sigma_1=(2,1)(4,2)$ (corresponding to the identity permutation in $\adj(M)_{J,I}$) and $\sigma_2=(2,2)(4,1)$ (corresponding to permutation $(12)(21)$ in $\adj(M)_{J,I}$).\end{exa}

Of course, we could finish here and report both of these assignments with supervisions as solutions, returning $S=\{\sigma_1, \sigma_2\}$ as the solution set.  But it may be that, given a subset $S$ of $k$ assignments with supervisions $I$ on $J$ which all have the same base value, it is desired that we return the 'best' of these under some criteria.
Obviously we could consider which of the supervisions in $S$ have the optimal value in $$M[I, J]=\begin{pmatrix}-3&0\\-1&-6\end{pmatrix}.$$  In this case the optimum is clearly given by $\sigma_2$.

However, it could be that this is not the best criteria to differentiate between supervisions.  Consider, for example, the following:  a low value of $M_{i,j}$ could indicate that worker $i$ has not been trained for job $j$, in which case supervising/training $i$ on $j$ could be beneficial to the company as it would increase that workers skill set and usefulness in the future.
In light of this, we consider a matrix $C$ reflecting the value of supervising $i\in I$ on task $j\in J$.  

\begin{df} 
\label{def:C}
We define $C\in\rmax^{k\times k}$ to be the \textit{priority matrix}, with rows $i\in I$ and columns $j\in J$, where $C_{i,j}$ is the priority the supervisor sets for supervising person $i\in I$ on task $j\in J$.   
Denoting by $E_{J,I}^{ap}$ the set of edges belonging to the optimal bijections 
in $\adj(M)_{J,I}$, we will assume that  
\begin{equation}
\label{e:essential-edges}
C_{i,j}\neq -\infty\Rightarrow (j,i)\in E_{J,I}^{ap},
\end{equation}
and that at least one bijection from $I$ to $J$ has a finite weight with respect to $C$.
\if{
C'_{i,j}=
\begin{cases}
C_{i,j}, &\text{if $(j,i)\in E_{JI}^{ap}$},\\
-\infty, &\text{otherwise.}
\end{cases}
\end{equation}  
}\fi
\end{df}

Note that by~\eqref{e:essential-edges},
finite priorities are assigned only for the edges in $E_{J,I}^{ap}$. Thus, assumption~\eqref{e:essential-edges} allows the supervisors to assign only essential priorities and it 
can decrease their costs of assigning priorities from $O(k^2)$ to $O(k)$ if there is only a small number 
of optimal permutations in $\ap(\adj(M)_{J,I})$. Let us consider the computational cost of 
identifying $E_{J,I}^{ap}$, which is necessary for setting the priorities matrix $C$.

\begin{pro}
Given $A\in\rmax^{n\times n}$ and $I,J\subseteq [n]$ with $|I|=|J|=k$, one needs no more than
$O(k^2n^3)$ operations to identify all edges of $E_{J,I}^{ap}$.
\end{pro}
\proof
We first compute $\adj(M)_{J,I}$. To compute each entry $\adj(M)_{j,i}$ we need to find an optimal
bijection in $M_{\{i\}^c,\{j\}^c}$. This takes $O((n-1)^3)$ operations 
for each entry, and hence $O(k^2n^3)$ operations for all $k^2$ edges.   

We now identify the set of edges $E_{J,I}^{ap}$ from $J$ to $I$ that are contained
in these bijections.  To do this we first apply the Hungarian algorithm
to $\adj(M)_{J,I}$. Having complexity $O(k^3)$, it brings matrix $\adj(M)_{J,I}$
to the form where all entries on optimal permutations are $\unit$ and all other entries do not 
exceed $\unit$. In order to identify all edges on optimal permutations we can further decrease the 
entries which do not lie on optimal permutations. This can be done, e.g. , by means of the strict visualization scaling of~\cite{scal} in no more that $O(k^3)$ operations.
\endproof

Note that in general, except for imposing~\eqref{e:essential-edges}, we 
do not assume that the entries of $M$ and $C$ are anyhow related to 
one another.

\if{
The supervisor could use the matrix $C$, or immediately $C'$ if the set $E_{JI}^{ap}$ is 
already known to them, to order the supervisions from most desirable to least desirable, giving high values to $C_{i,j}$ (or $C'_{i,j}$) for which supervising $i$ on $j$ is important.  Then, {if} a set of $k$ assignments with supervisions $I$ on $J$ and optimal base value does have $i$ supervised on $j$, this set of assignments is more likely to be rated as best.
}\fi

In other words, we now want to rate a set of $k$ assignments with supervisions under two criteria.
First, the set must have optimal base value with respect to assignment  matrix $A$.
Then, of those that meet this criteria, we can choose to optimise with respect to the supervision priority matrix $C$.

\begin{exa} 
\label{Cexa-1}
Recalling Example~\ref{Cexa}, suppose that it is highly desired that person $2$ is supervised on task $1$, it is recommended that person $2$ is 
supervised on task $2$  and that person $4$ is supervised on task $1$, however person $4$ is well-trained on task $2$. Then we might set 
\begin{equation}
\label{e:Cexa-1}
C=\begin{pmatrix}3&1\\1&0\end{pmatrix}
\end{equation}
and conclude that the overall most valued supervisions are $2$ on $1$ with 
$4$ on $2$ (regardless of what we have in $M$ and although $4$ is well-trained on $2$).
 That is, the supervision $\sigma_1$ is optimal. \end{exa}

\if{
\begin{rem} It might be that the optimal permutation in $C$ does not have an optimal base value with respect to $A$. In that case 
$\per(C)\notin S$, that is, the optimal supervision is not simply $\per(C)$. 

Also, given supervisions in $S$ which link to an optimal $(1,k)$-regular multigraph, one may recover the  set of $k$ assignments  with optimal base value in the following way: 
Let $\sigma\in S$, and $\pi_1,\dots,\pi_k\in S_n$ the optimal assignments associated to $\sigma$. 
 For each edge $(i_t, \sigma(i_t)):\ t\in[k]$, take the bijections $\beta_t$ corresponding to an optimal permutation of 
$A[[n]\setminus\{i_t\}, [n]\setminus\{\sigma(i_t)\}]$, i.e. to the value of $\adj(A)_{\sigma(i_t),i_t}$. 
 Then $\pi_t$ is exactly $\beta_t$ together with the assignment $i_t$ to $\sigma(i_t)$.\end{rem}
}\fi

Let us now formally define a general version of the optimization problem which we solved
in Example~\ref{Cexa-1}.

\begin{df} Given an assignment  matrix $M\in\rmax^{n\times n}$ and a supervision priority matrix $C\in\rmax^{k\times k}$, the optimal value of $k$ assignments 
with supervisions $I\subset[n]$ on tasks $J\subset[n]$ prioritized by $C$ is defined as 
\begin{equation}
\label{e:prioritized}
\adj(M)^{\wedge k}_{J,I}\odot \per(C)=\per(\adj(M)_{J,I})\odot\per(C),
\end{equation}
where $C$ is as in Definition~\ref{def:C}. 

\end{df}

\begin{obs}
For the value defined in~\eqref{e:prioritized} we have
\begin{equation}
\label{e:prioritized-mod}
\begin{split}
\per(\adj(M)_{J,I})\odot\per(C) &= 
\bigoplus_{\sigma'\in S_{J,I}}w(\sigma', \adj(M)_{J,I})\odot\bigoplus_{\sigma\in S_{J,I}} 
w(\sigma^{-1},C)\\
&=\bigoplus_{\sigma'\in S_{J,I}}w(\sigma', \adj(M)_{J,I})\odot\bigoplus_{\sigma\in\ap(\adj(A)_{J,I})}
w(\sigma^{-1}, C).
\end{split}
\end{equation}
This value is attained by those $k$ assignments with supervisions $I$ of $J$ for which the
total value of these supervisions (computed from the priority matrix $C$) 
is the greatest.
\end{obs}

\begin{pro}
Given $M\in\rmax^{n\times n}$, $I,J\subseteq [n]$, set $E_{JI}^{ap}$ and 
$C\in\rmax^{k\times k}$ that satisfies~\ref{e:essential-edges}
one needs at most $O(k^3)$ operations to compute value~\eqref{e:prioritized} and
no more than $O(kn^3)$ operations to identify a set of supervised assignments that attains that value.
\end{pro}
\proof

We first compute 
$\per(C)$ in at most $O(k^3)$ operations  and identify a bijection  
$\beta\in\ap(C)$. By~\eqref{e:essential-edges}
$(\beta(i),i)\in E_{J,I}^{ap}$ for each $i\in I$, hence $\beta^{-1}\in\ap(\adj(M)_{J,I})$ and 
we compute $w(\beta^{-1})=\per(\adj(M)_{J,I})$ in $O(k)$ operations. 

By now, value~\eqref{e:prioritized} has been found. In order to find a supervised assignment that 
attains that value we take $\beta\in\ap(C)$ found previously, and for each edge $(i,j)$ from that bijection
we find an optimal bijection in $M_{\{i\}^c,\{j\}^c}$. This takes $O((n-1)^3)$ operations for each edge, and hence $O(kn^3)$ operations for all $k$ edges of the bijection.

\endproof

\begin{rem}
\label{r:improvecompl}
In the procedure described above, it is possible to decrease the amount of operations required to 
compute $\per(C)$. Indeed, let $m$ be the number of entries that are not equal to $-\infty$
in $C$. Applying the Fibonacci heaps technique of~\cite{FT87}, the complexity of
computing $\per(C)$ is decreased to $O(k^2\log k+km)$.
\end{rem}

\begin{exa}
\label{Cexa-2}
In Example~\eqref{Cexa-1} we found that $\sigma_1$ is optimal  
if we take $C$ as in~\eqref{e:Cexa-1}.
For the assignment $2\rightarrow 1$ of $\sigma_1$ we consider $$M[[4]\setminus\{2\}, [4]\setminus\{1\}]=\begin{pmatrix}{\bf 1}&-2&-4\\4&0&{\bf 6}\\-6&{\bf 3}&0
\end{pmatrix},$$ which has optimal permutation shown in bold, corresponding to the bijection $\beta_{1}=(1,2)(3,4)(4,3)$.
By adding in the supervision edge, we recover $\pi_1=(1\ 2)(2\ 1)(3\ 4)(4\ 3)$.
Similarly, for the supervision $4\rightarrow 2$, we get the bijection $\beta_2=(1,1)(2,3)(3,4)$ which corresponds to $\pi_2=(1)(2\ 3\ 4).$
\end{exa}

\if{
\begin{alg}\label{Alg1}\ 

INPUT:  Assignment  matrix $A\in\rmax^{n\times n}$, $k$, $I,J\subset[n]:\ |I|=k=|J|$, supervision priority matrix $C\in\rmax^{k\times k}$.

OUTPUT: An optimum set of assignments~$\pi_1,\dots,\pi_k\in S_n$\text{ with supervision~}$\tau\in S_{I,J}$.

Note:  Maintain original row and column indexing throughout the algorithm.

\begin{enumerate}
\item Calculate $\adj(A)$ and keep track of one optimal bijection $\beta(i,j)$ for each entry $\adj(A)_{ji}$.

\item Use Hungarian algorithm to transform $\adj(A)_{J,I}$ to $B\le 0$ with a maximal set of independent zeros.
Determine $\operatorname{ap}(B)$ and $|\operatorname{ap}(B)|$ from this set. 

\item If $|\operatorname{ap}(B)|=1$, then let $\sigma$ be the bijection $I\mapsto J$ corresponding to the unique optimal permutation in $B$. Go to 5. 

\item  Otherwise, define $C'\in\rmax^{k\times k}$ with entries $$C'_{ij}=\begin{cases}C_{ij}&;\ \text{if } B_{ji}=0,\\ -\infty&;\ \text{otherwise.} \end{cases}$$
Find a bijection $\sigma$ corresponding to an optimal permutation in $C'$.

\item  For each $i\in [k]$ let $\pi_i$ be the permutation formed of $\beta(i,\sigma(i))$ and the assignment $i$ to $\sigma(i)$.
\end{enumerate}\end{alg}

\todo[inline]{A: As it is now, the supervisor should set priorities to all pairs i,j. That can be quite a lot assignments, for 
example 20x20, where eventually there might only be for example two optimal supervisions, which sums up to 40 assignments... 
Why not  preform steps 1-3 without setting C, and then set it in step 4? (that is C' is C, there is no need for this mediating matrix)} 

\begin{thm}Algorithm \ref{Alg1} is correct and runs in $\mathcal{O}(n^5)$ time.\end{thm}
\proof 

Correctness:

It is clear that $\adj(A)_{j,i}$ corresponds to a bijection $\beta(i,j): N\setminus\{i\}\rightarrow N\setminus\{j\}$.  Further, adding $i\rightarrow j$ to $\beta(i,j)$ creates a permutation $\pi(i,j)\in S_n$ of maximum weight which includes the assignment $i\rightarrow j$.

Then any optimal permutation $\sigma\in S_k$ of $\adj(A)_{J,I}$ maximizes $\sum_{i=1}^k\beta(\sigma(i), i)$.  To optimize $\sum_{i=1}^k \pi(\sigma(i),i)$ with respect to this, it remains to optimize $\sum_{i=1}^k C_{\sigma(i),i}$.  

By the Hungarian method, $B$ has the same set of maximum weighted permutations as $\adj(A)_{J,I}$ and further, these permutations can be identified by the zero entries, i.e. 
$$w(\pi, B)=0\ge w(\tau, B)\Leftrightarrow B_{i,\pi(i)}=0\ \forall i.$$  

Thus, since $C'$ only preserves weights of entries corresponding to zeros in $B^T$,  it follows that any permutation in $C'$ will be an optimal permutation in $B^T$, and therefore reflects an optimal permutation (after reversing) in $\adj(A)_{J,I}$.

Complexity: 

Step 1 requires $n^2$ applications of Hungarian algorithm on $(n-1)\times (n-1)$ matrix and thus has complexity $\mathcal{O}(n^5)$. Step 2 needs only one application of Hungarian algorithm and has 
complexity $\mathcal{O}(n^3)$, and Step 3 has negligible complexity. On Step 4, the same algorithm is applied to a $k\times k$ matrix and 
this adds only $\mathcal{O}(k^3)$ complexity, and complexity of Step 5 is again negligible. Thus the complexity is 
dominated by Step 1 ($\mathcal{O}(n^5)$) and this completes the proof.
\endproof

%%%%%%%%%%%%%%%%%%%%%%%%%%%%%%%%%%%%%%%%%%%%%%%%%%%%%%

}\fi

 \section{Jacobi identity}

The following theorem is the tropical analogue of the Jacobi identity (see~\cite{{Fallat&Johnson}}), and was recently proved in~\cite{AGN}.
\begin{thm}\label{JI} (Jacobi identity,~\cite[Theorem~3.4]{AGN})\label{thm:JacobiIdentity} Let $M\in\R_{\max}^{n\times n}$ and  
$I,J\subseteq[n]$ such that~$|I|=|J|=k$. For every $k\in \{0\}\cup [n]$, at least one of the following statements holds \begin{enumerate}
\item $[\per(M)^{\odot -1}\adj(M)]^{\wedge k}_{_{I,J}}=\per(M)^{\odot -1}
M^{\wedge n-k}_{_{J^c, I^c}},$  

\item There exist distinct bijections  $ \pi, 
\sigma\in S_{I,J}$  such that~
$$[\adj(M)]^{\wedge k}_{_{I,J}}=\sum_{i\in I}\adj(M)_{i,\pi(i)}=\sum_{i\in I}
\adj(M)_{i,\sigma(i)}.$$\end{enumerate}\end{thm}

Actually, the identity proved in \cite{AGN} is a true analogue of the original, for it is proved over an extension of the 
tropical semiring  called  \textit{Symmetrized} (see~\cite{G.Ths}). The Symmetrized tropical semiring  is constructed by copies of
 $R_{\max}$ to include so called `balancing' elements, and therefore 
include the signs of permutations throughout (in $\per$, $\adj$ and $\wedge k$). Since in this paper we work over the 
(non-extended) tropical semiring, the signs are omitted from the formulation stated in Theorem~\ref{JI}.

\if{
Let us also formulate a version of this identity when $\per(M)=\unit$.

\begin{cor}
\label{c:perunit}
 If $\per(M)=\unit$ then the Jacobi identity becomes
\begin{enumerate}
\item $[\adj(M)]^{\wedge k}_{_{I,J}}=M^{\wedge n-k}_{_{J^c, I^c}},$  

\noindent or

\item There exist distinct bijections  $ \pi, 
\sigma\in S_{I,J}$  such that~
$$[\adj(M)]^{\wedge k}_{_{I,J}}=\sum_{i\in I}\adj(M)_{i,\pi(i)}=\sum_{i\in I}
\adj(M)_{i,\sigma(i)}.$$\end{enumerate}\end{cor}
}\fi

The main purpose of this section will be to obtain 
a graph-theoretic counterpart of Theorem~\ref{JI}.

Let us consider the following kind of matrices.

\begin{df}
$M\in\rmax^{n\times n}$ is called {\bf  normalized} if $\Id$ is 
an optimal permutation of $M$ (i.e., $\Id\in\ap(M)$) and $M_{i,i}=\unit$ for all $i\in [n]$.

$M$ is called {\bf strictly normalized} if it is normalized and $\Id$ is the only optimal permutation
(i.e., $\ap(M)=\{\Id\}$).
\end{df}

\if{
\begin{df}
$M\in\rmax^{n\times n}$ is called {\bf  (tropically) diagonally dominant} if the identity is 
an optimal permutation of $M$, i.e., if $\Id\in\ap(M)$.

$M$ is called {\bf strongly (tropically) diagonally dominant} if the identity is 
the only optimal permutation of $M$, i.e., if $\ap(M)=\{\Id\}$.
\end{df}
}\fi

We will use the following result, whose proof will be skipped.

\begin{lem}\label{norm}
If $M\in\R_{\max}^{n\times n}$ is normalized (respectively, strictly normalized) then 
the weight of every cycle 
is not greater (respectively, smaller) than  the weight of the identity permutation on its subset of nodes, and therefore the weight of every
permutation is not greater (respectively, is smaller) than that of a permutation obtained by replacing either of its cycles
by the corresponding identity permutation.
\end{lem}

\if{
\begin{lem}\label{norm}
If $M\in\R_{\max}^{n\times n}$ is diagonally dominant
(respectively, strictly diagonally dominant) then 
the weight of every cycle 
is not greater (respectively, smaller) than  the weight of the identity permutation on its subset of nodes, and therefore the weight of every
permutation is not greater (respectively, is smaller) than that of a permutation obtained by replacing either of its cycles
by the corresponding identity permutation.
\end{lem}
}\fi

We now consider optimal $(1,k)$-regular multigraphs of the digraph associated with 
a normalized matrix.

\begin{nota}\label{thmnota}
Let~$G=([n],E)$ be a simple weighted digraph with~$n$ nodes associated with a normalized matrix, i.e.,
$\Id$ is an optimal
permutation and all the loops (i.e., edges $e$ with $s(e)=t(e)$) are in $E$ and have weight $\unit$.
Let us construct an optimal~$(1,k)$-regular multigraph  of~$G$ with 
respect to~$I,J\subseteq[n]$. We will denote it by $F=\big([n],\biguplus_{i\in[k]}E_{\rho_i},\pi\big)$

Assume~$\pi\ne \Id$, and let 
\begin{equation}\label{N1}e_i\text{ be the edge of~$E_\pi$ which is in~$\rho_i$, }\end{equation} 
\begin{equation}\label{N2}\text{$C_i$ be the cycle in~$\rho_i$ which includes~$e_i$,}\end{equation}
\begin{equation}\label{N3}\text{$P_i$ be the elementary path such that $C_i=P_i\circ e_i$.}\end{equation}
 One can find $\rho_i, C_i, P_i, e_i$ described
 in Figures~\ref{P&C} and~\ref{EP&C} in case that $\rho_i$ has no more than one cycle which is not a loop.

\begin{figure}[h]
\begin{center}\begin{tikzpicture}
[main_node/.style={->,>=stealth',shorten >=1pt,circle,fill=black,minimum size=0.8em,inner sep=1pt,line width=0.3mm},
sx_node/.style={->,>=stealth',shorten >=1pt,circle,fill=white,minimum size=0.8em,inner sep=1pt,line width=0.3mm},
sec_node/.style={->,>=stealth',shorten >=1pt,circle,out=30,in=150,fill=white,minimum size=0.8em,inner sep=1pt,line width=0.3mm},
thr_node/.style={->,>=stealth',shorten >=1pt,circle,out=-150,in=-30,fill=black,minimum size=0.8em,inner sep=1pt,line width=0.3mm},
fr_node/.style={->,>=stealth',shorten >=1pt,circle,out=150,in=270,fill=black,minimum size=0.8em,inner sep=1pt,line width=0.3mm},
ft_node/.style={->,>=stealth',shorten >=1pt,circle,out=-30,in=90,fill=black,minimum size=0.8em,inner sep=1pt,line width=0.3mm},
prl_node/.style={-,>=stealth',shorten >=1pt,circle,out=240,in=120,fill=black,minimum size=0.8em,inner sep=1pt,line width=0.5mm},
prr_node/.style={-,>=stealth',shorten >=1pt,circle,out=-60,in=60,fill=black,minimum size=0.8em,inner sep=1pt,line width=0.5mm}]

\node[draw=white](m13) at (-4,4) {$\color{black}\rho_i:$};

\node[main_node](m11) at (-2,3.75) {$\color{black}$};
\node[draw=white](m13) at (-1,3.75) {$\color{black}\cdots$};
\node[main_node](m12) at (0,3.75) {$\color{black}$};

\node[draw=white](m13) at (-1,3.25) {$\color{black}\underbrace{\ \ \ \ \ \ \ \ \ \ \ \ \ \ \ \ }$};
\node[draw=white](m13) at (-1,2.75) {$\color{black}\text{loops on }[n]\setminus\overline{s(C_i)}\ \ $};

 \draw[sx_node]  (m11) to [out=130,in=50,looseness=8] (m11);  
 \draw[sx_node]  (m12) to [out=130,in=50,looseness=8] (m12);  

\node[main_node](c1) at (2,4.5) {$\color{black}$};
\node[main_node](c2) at (2,3) {$\color{black}$};
\draw[main_node]  (c2) edge node{} (c1);
\node[main_node](c3) at (3.5,5) {$\color{black}$};
\draw[main_node,color=red]  (c1) edge node{$\begin{array}{c}e_i\in E_\pi\\\\\\\end{array}$} (c3);
\node[main_node](c4) at (4.5,4) {$\color{black}$};
\draw[main_node]  (c3) edge node{} (c4);
\node[main_node](c5) at (3.75,2.75) {$\color{black}$};
\draw[main_node]  (c4) edge node{} (c5);
\draw[main_node]  (c5) edge node{} (c2);

\node[draw=white](m13) at (3,2.25) {$\color{black}C_i$};

\end{tikzpicture}\end{center}\caption{Paths and cycles}\label{P&C}\end{figure}
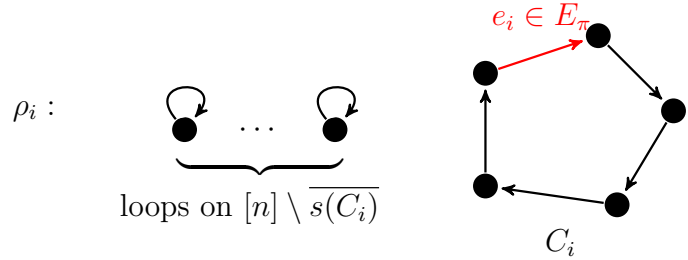

\begin{figure}[h]
\begin{center}\begin{tikzpicture}
[main_node/.style={->,>=stealth',shorten >=1pt,circle,fill=black,minimum size=0.8em,inner sep=1pt,line width=0.3mm},
sx_node/.style={->,>=stealth',shorten >=1pt,circle,fill=white,minimum size=0.8em,inner sep=1pt,line width=0.3mm},
sec_node/.style={->,>=stealth',shorten >=1pt,circle,out=30,in=150,fill=white,minimum size=0.8em,inner sep=1pt,line width=0.3mm},
thr_node/.style={->,>=stealth',shorten >=1pt,circle,out=-150,in=-30,fill=black,minimum size=0.8em,inner sep=1pt,line width=0.3mm},
fr_node/.style={->,>=stealth',shorten >=1pt,circle,out=150,in=270,fill=black,minimum size=0.8em,inner sep=1pt,line width=0.3mm},
ft_node/.style={->,>=stealth',shorten >=1pt,circle,out=-30,in=90,fill=black,minimum size=0.8em,inner sep=1pt,line width=0.3mm},
prl_node/.style={-,>=stealth',shorten >=1pt,circle,out=240,in=120,fill=black,minimum size=0.8em,inner sep=1pt,line width=0.5mm},
prr_node/.style={-,>=stealth',shorten >=1pt,circle,out=-60,in=60,fill=black,minimum size=0.8em,inner sep=1pt,line width=0.5mm}]

\node[draw=white](m13) at (-4,4) {$\color{black}\rho'_i:$};

\node[main_node](m11) at (-2,3.75) {$\color{black}$};
\node[draw=white](m13) at (-1,3.75) {$\color{black}\cdots$};
\node[main_node](m12) at (0,3.75) {$\color{black}$};

\node[draw=white](m13) at (-1,3.25) {$\color{black}\underbrace{\ \ \ \ \ \ \ \ \ \ \ \ \ \ \ \ }$};
\node[draw=white](m13) at (-1,2.75) {$\color{black}\text{loops on }[n]\setminus V(E_{P_i})\ \ $};

 \draw[sx_node]  (m11) to [out=130,in=50,looseness=8] (m11);  
 \draw[sx_node]  (m12) to [out=130,in=50,looseness=8] (m12);  

\node[main_node](c1) at (2,4.5) {$\color{black}$};
\node[](c) at (1.5,5) {$\color{black}s(e_i)=t(P_i)$};
\node[main_node](c2) at (2,3) {$\color{black}$};
\draw[main_node]  (c2) edge node{} (c1);
\node[main_node](c3) at (3.5,5) {$\color{black}$};
\node[](c) at (4,5.5) {$\color{black}s(P_i)=t(e_i)$};
\node[main_node](c4) at (4.5,4) {$\color{black}$};
\draw[main_node]  (c3) edge node{} (c4);
\node[main_node](c5) at (3.75,2.75) {$\color{black}$};
\draw[main_node]  (c4) edge node{} (c5);
\draw[main_node]  (c5) edge node{} (c2);

\node[draw=white](m13) at (3,2.25) {$\color{black}P_i$};

\end{tikzpicture}\end{center}\caption{Elementary paths, sources and targets}\label{EP&C}\end{figure}
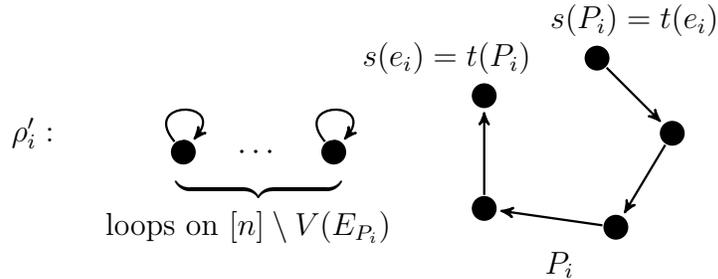

\end{nota}

Let us also add some remarks on Notation~\ref{thmnota}.

\begin{rem}
For every~$i\in[k]$, denote by $\rho'_i$ the subdigraph~$([n],E_{\rho_i}\setminus e_i)$. 
Note that it is a permutation: $\rho'_i\in S_{\{s(e_i)\}^c,\{t(e_i)\}^c}$.
 Note also that $P_i$ introduced in~\eqref{N3} has~$s(P_i)=t(e_i)$  
(and $s(e_i)=t(P_i)$), and that it is the only elementary path in the decomposition of $\rho'_i $. See Figure~\ref{EP&C}.
\end{rem}

\begin{rem}
Note that
\begin{itemize} 
\item $\overline{s(C_i)}=\overline{s(P_i)}\cup t(P_i)=V(E_{P_i})$,  

\item All 
the loops on $[n]\setminus(\bigcup_{i\in[k]} \overline{s(C_i)})$ belong to $F$ if all 
$\rho_i$ are as on Figure~\ref{P&C}.

\item Some (but not all) $e_i$ might be loops. In this case $C_i=e_i$, and $E_{P_i}=\emptyset$.
\end{itemize}
\end{rem}

\begin{rem}
\label{r:strongly}
If identity is the unique optimal permutation (i.e., if the associated matrix is strictly
normalized), then all permutations in any optimal $(1,k)$-regular multigraph have no more than
one cycle in their decomposition (as in Figure~\ref{P&C}).
\end{rem}

\begin{df}\label{equiv}
For a given multigraph $F$, denote its set of loops by $L_F$.
A multigraph $F$ is said to be \textit{equivalent} to a simple graph $G$ if 
$E_G=E_F\setminus L_F$. \end{df}

The following theorem 
is (loosely speaking) an equivalent of  the Jacobi identity (\cite{AGN}) in terms of graphs theory. 
Recall that the $i,j$-entry of the adjoint matrix of a matrix $A$ is determined by permutations in $S_n$ where $j$ is sent to $i$.
The entry $A_{j,i}$ is then removed from the product resulting in this entry. The corresponding graph of this product is therefore 
missing the $j,i$-edge.
These `missing $\adj$-edges' are combinatorially considered as  fixed conditions. 
(It resembles the combinatorial approach explaining  the identity for the number of  choices of a subset of  size $k$ from a set of size~$n$:
$${{n-1}\choose{k-1}}+{{n-1}\choose{k}}=
{{n}\choose{k}}$$
as choosing from a set of size $n-1$ having a fixed element either inside the subset or outside the subset.)

\begin{thm}\label{thm:JacobiGraph}  Let $G=([n],E)$ be a simple weighted digraph
where $\Id$ is an optimal permutation.  
Let $F=\big([n],\biguplus_{i\in[k]}E_{\rho_i},\pi\big)$ be an optimal~$(1,k)$-regular multigraph  of~$G$ with respect to 
$I,J\subseteq [n]$, such that each $\rho_i$ has no more than one cycle that is not a loop.  %as in Notation~\ref{thmnota}.
Then there exists a $k$-regular multigraph~$F'=\big([n],\biguplus_{i\in[k]}E_{\tau_i}\big)$ of~$G$ such that  
at least one of the following statements holds 
\begin{enumerate} 
\item  $F'=F$ and 
$E_\pi\subseteq E_{\tau_\ell}$ for some $\ell\in [k]$, and
$\pi'\in S_{I^c,J^c}$ defined by $E_{\pi'}=E_{\tau_\ell}\setminus E_\pi $  
satisfies $\bigcup_{i\in[k]} E_{P_i}\subseteq E_{\pi'}$ and is an optimal bijection (with respect to $I^c, J^c$). 

\item There exists $\tilde{\pi}\in S_{I,J}$ such that
 $F'=\big([n],\biguplus_{i\in[k]}E_{\tau_i},\tilde{\pi}\big)\neq F$ satisfies $$(\biguplus_{i\in[k]}E_{\tau_i})\setminus (E_{\tilde{\pi}}\cup L_{F'})\subseteq (\biguplus_{i\in[k]}E_{\rho_i})\setminus E_\pi$$ and is also an optimal $(1,k)$-regular multigraph with respect to $I,J$.
\end{enumerate} 
\end{thm}

\begin{proof} 
Let us first observe that the optimality of a $(1,k)$-regular multigraph with respect to $I,J\in[n]$ 
or a bijection in $S_{I,J}$ will not change if we multiply
each column (or each row) of the matrix associated with $G$ by some scalar. 
Indeed, this will multiply the weight of each $(1,k)$-regular multigraph
by the same constant and the weight of each bijection in $S_{I,J}$ by the same constant. This shows that 
it suffices to prove the theorem for the case where in addition to $\Id$ being optimal, 
all loops are in $E$ and have weight $\unit$, i.e., where the associated matrix is normalized. 
We will further assume that $G$ has this property and consider the following two 
principal cases.

\textbf{Case~1:  All paths $P_i$ for $i\in [k]$ are
pairwise disjoint.}

We will show that (1) occurs by proving that $\biguplus_{i\in[k]}E_{\rho_i}\setminus E_\pi$ is equivalent (in the sense of Definition~\ref{equiv}) to an optimal permutation  w.r.t.~$I^c,J^c$ (see Figure~\ref{case1}).
	
\begin{figure}[h]\begin{center}\begin{tikzpicture}
[main_node/.style={->,>=stealth',shorten >=1pt,circle,fill=black,minimum size=0.8em,inner sep=1pt,line width=0.3mm},
sx_node/.style={->,>=stealth',shorten >=1pt,circle,fill=white,minimum size=0.8em,inner sep=1pt,line width=0.3mm},
sec_node/.style={->,>=stealth',shorten >=1pt,circle,out=30,in=150,fill=white,minimum size=0.8em,inner sep=1pt,line width=0.3mm},
thr_node/.style={->,>=stealth',shorten >=1pt,circle,out=-150,in=-30,fill=black,minimum size=0.8em,inner sep=1pt,line width=0.3mm},
fr_node/.style={->,>=stealth',shorten >=1pt,circle,out=150,in=270,fill=black,minimum size=0.8em,inner sep=1pt,line width=0.3mm},
ft_node/.style={->,>=stealth',shorten >=1pt,circle,out=-30,in=90,fill=black,minimum size=0.8em,inner sep=1pt,line width=0.3mm},
prl_node/.style={-,>=stealth',shorten >=1pt,circle,out=250,in=110,fill=black,minimum size=0.8em,inner sep=1pt,line width=0.5mm},
prr_node/.style={-,>=stealth',shorten >=1pt,circle,out=-70,in=70,fill=black,minimum size=0.8em,inner sep=1pt,line width=0.5mm}]

\node (11) at (-7.4,4) {$\color{black}(1,k)-$};
\node(11) at (-7.4,3.5) {regular};
\node (11) at (-7.4,3) {$\color{black}{\rho_{1}}$};
\node[draw=black,circle,white](21) at (-7.4,2) {$\color{black}\rho_{2}$};
\node (31) at (-7.4,1) {$\color{black}\vdots$};
\node (31) at (-7.4,0) {$\color{black}\rho_{k}$};

\node (s3) at (-7.26,-0.5) {$\color{black}$};
\node (s1) at (-7.26,3.5) {$\color{black}$};

\node (s2) at (-7.55,3.5) {$\color{black}$};
\node (s4) at (-7.55,-0.5) {$\color{black}$};

\draw[prr_node]   (s1) edge node{} (s3);
\draw[prl_node]   (s2) edge node{} (s4);

\node (11) at (-6.5,3) {$\color{black}=$};
\node (21) at (-6.5,2) {$\color{black}=$};
\node (31) at (-6.5,1) {$\color{black}\vdots$};
\node (31) at (-6.5,0) {$\color{black}=$};

\node (11) at (-5.2,3) {$ {\rho_{1}|_{\overline{s(e_1)}^c}}$};
\node (21) at (-5.2,2) {$ {\rho_{2}|_{\overline{s(e_2)}^c}}$};
\node (31) at (-5.2,1) {$ \vdots$};
\node (31) at (-5.2,0) {$ {\rho_{k}|_{\overline{s(e_k)}^c}}$};

\node (s3) at (-4.6,-0.5) {$\color{black}$};
\node (s1) at (-4.6,3.5) {$\color{black}$};

\node (s2) at (-5.8,3.5) {$\color{black}$};
\node (s4) at (-5.8,-0.5) {$\color{black}$};

\draw[prr_node]   (s1) edge node{} (s3);
\draw[prl_node]   (s2) edge node{} (s4);

\node (11) at (-4,3) {$\color{black}\circ$};
\node (21) at (-4,2) {$\color{black}\circ$};
\node (31) at (-4,1) {$\color{black}\vdots$};
\node (31) at (-4,0) {$\color{black}\circ$};

\node (s3) at (-3,-0.5) {$\color{black}$};
\node (s1) at (-3,3.5) {$\color{black}$};

\node (s2) at (-3.4,3.5) {$\color{black}$};
\node (s4) at (-3.4,-0.5) {$\color{black}$};

\draw[prr_node]   (s1) edge node{} (s3);
\draw[prl_node]   (s2) edge node{} (s4);

\node(11) at (-3.2,4) {$S_n$};
\node(11) at (-3.2,3.5) {$\lin$};
\node (11) at (-3.2,3) {$\color{black}C_1$};
\node (21) at (-3.2,2) {$\color{black}C_2$};
\node (31) at (-3.2,1) {$\color{black}\vdots$};
\node (31) at (-3.2,0) {$\color{black}C_k$};
\node (31) at (-3.2,-0.55) {$\circ$};
\node (32) at (-3.2,-0.8) {${\text{loops}}$};
\node (32) at (-3.2,-1.2) {${\text{on~\eqref{**}}}$};

\node (11) at (-2.3,3) {$\color{black}=$};
\node (21) at (-2.3,2) {$\color{black}=$};
\node (31) at (-2.3,1) {$\color{black}\vdots$};
\node (31) at (-2.3,0) {$\color{black}=$};

\node (s3) at (-1.2,-0.5) {$\color{black}$};
\node (s1) at (-1.2,3.5) {$\color{black}$};

\node (s2) at (-1.6,3.5) {$\color{black}$};
\node (s4) at (-1.6,-0.5) {$\color{black}$};

\draw[prr_node]   (s1) edge node{} (s3);
\draw[prl_node]   (s2) edge node{} (s4);

\node (31) at (-1.4,4) {$\Id^{k-1}$};
\node (32) at (-1.4,3.5) {$\lfeq$};
\node (11) at (-1.4,3) {$_{\text{loops}}$};
\node (21) at (-1.4,2) {$_{\text{loops}}$};
\node (31) at (-1.4,1) {\vdots};
\node (31) at (-1.4,0) {$_{\text{loops}}$};

\node (11) at (-0.6,3) {$\color{black}\circ$};
\node (21) at (-0.6,2) {$\color{black}\circ$};
\node (31) at (-0.6,1) {$\color{black}\vdots$};
\node (31) at (-0.6,0) {$\color{black}\circ$};

\node (s3) at (0.2,-0.5) {$\color{black}$};
\node (s1) at (0.2,3.5) {$\color{black}$};

\node (s2) at (0,3.5) {$\color{black}$};
\node (s4) at (0,-0.5) {$\color{black}$};

\draw[prr_node]   (s1) edge node{} (s3);
\draw[prl_node]   (s2) edge node{} (s4);

\node(11) at (0.1,4.6) {$S_{I^c,J^c}$};
\node(11) at (0.1,4.2) {$\lin$};
\node(11) at (0.1,3.9) {$\pi'$};
\node(11) at (0.1,3.5) {$\lfeq$};
\node (11) at (0.1,3) {$\color{black}P_1$};
\node (21) at (0.1,2) {$\color{black}P_2$};
\node (31) at (0.1,1) {$\color{black}\vdots$};
\node (31) at (0.1,0) {$\color{black}P_k$};
\node (31) at (0.1,-0.55) {$\circ$};
\node (32) at (0.1,-0.8) {${\text{loops}}$};
\node (32) at (0.1,-1.2) {${\text{on~\eqref{**}}}$};

\node (11) at (0.8,3) {$\color{black}\circ$};
\node (21) at (0.8,2) {$\color{black}\circ$};
\node (31) at (0.8,1) {$\color{black}\vdots$};
\node (31) at (0.8,0) {$\color{black}\circ$};

\node (s3) at (1.6,-0.5) {$\color{black}$};
\node (s1) at (1.6,3.5) {$\color{black}$};

\node (s2) at (1.4,3.5) {$\color{black}$};
\node (s4) at (1.4,-0.5) {$\color{black}$};

\draw[prr_node]   (s1) edge node{} (s3);
\draw[prl_node]   (s2) edge node{} (s4);

\node(11) at (1.5,4.6) {$S_{I,J}$};
\node(11) at (1.5,4.2) {$\lin$};
\node(11) at (1.5,3.9) {$\pi$};
\node(11) at (1.5,3.5) {$\lfeq$};
\node (11) at (1.5,3) {$\color{black}e_1$};
\node (21) at (1.5,2) {$\color{black}e_2$};
\node (31) at (1.5,1) {$\color{black}\vdots$};
\node(31) at (1.5,0) {$\color{black}e_k$};

\node (a1) at (3,1.5) {$\color{black}$};
\node (a2) at (2,1.5) {$\color{black}$};

\draw[main_node]   (a2) edge node{} (a1);

\node (11) at (3.9,3) {$\color{black}\tau_1=\Id$};
\node (21) at (4,2) {$\color{black}\vdots$};
\node (31) at (4.1,1) {$\color{black}\tau_{k-1}=\Id$};
\node (31) at (4.2,0) {$\color{black}\tau_k=\pi'\circ\pi$};
\end{tikzpicture}\end{center}\caption{Case~(1): Optimal $(1,k)$-regular multigraph $F$ is equivalent to an optimal permutation  w.r.t.~$I^c,J^c$. (See Notations~\eqref{N1}\ --\ \eqref{N3})}\label{case1}\end{figure}
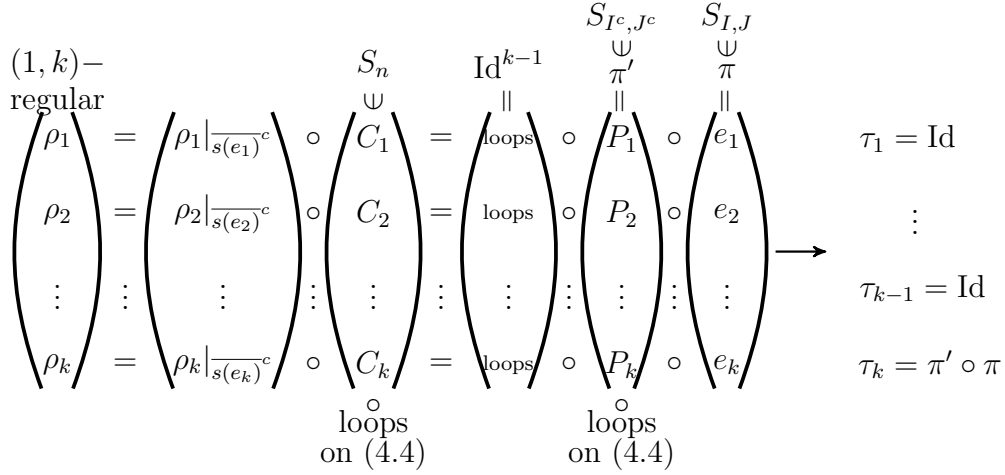

To formalize the description of this case, which will later help us to deal with the remaining 
cases, observe that $C_i$ (equivalently $P_i$. See Notation~\ref{N2}) are disjoint exactly when all the following conditions hold:

\begin{enumerate}[\ \ \ (a)]
\item\label{(a)} All sources and targets are disjoint, which is equivalent to
$$s(P_i)\in J\setminus I=I^c\setminus J^c\text{ and }t(P_i)\in I\setminus J=J^c\setminus I^c,\ \ \forall i\in [k], $$
\item\label{(b)} Sources and targets are disjoint to all intermediate nodes, which is equivalent to
$$V(E_{P_i})\setminus\{s(P_i),t(P_i)\}\subseteq I^c\cap J^c,\ \ \forall i\in [k],$$
\item\label{(c)} All intermediate nodes are disjoint:
 $$(V(E_{P_i})\setminus\{s(P_i),t(P_i)\})\bigcap(V(E_{P_j})\setminus\{s(P_j),t(P_j)\})=\emptyset,\ \ \forall i\ne j.$$ 
\end{enumerate} 

Then, under these  conditions, the composition of~$C_1\circ\cdots\circ C_k$ with the loops on 
\begin{equation}\label{**}[n]\setminus\big(\underbrace{\bigcup_{i\in[k]}
V(E_{P_i})}_{\supseteq I\cup J}\big)=(J\cup I)^c\setminus\big(\bigcup_{i\in[k]}V(E_{P_i})\big)=J^c\cap I^c\setminus\big(\bigcup_{i\in[k]}{V(E_{P_i})}\big)\end{equation}
is a permutation that can be taken for $\tau_{\ell}$  
and the composition of~$P_1\circ\cdots\circ P_k$  
with these loops  for the 
bijection~$\pi'\in S_{I^c, J^c}$ such that  $E_{\pi'}=E_{\tau_i}\setminus E_{\pi}$.
Note that all the edges of $(\biguplus_{i\in[k]}E_{\rho_i})\setminus E_\pi$ that are not in $E_{\tau_{\ell}}$
are loops that compose $k-1$ copies of the identity permutation. So we can take 
$F'=\big([n],\biguplus_{i\in[k]}E_{\tau_i}\big)$, where $\tau_{\ell}$ is as defined above and 
all other $\tau_i$ for $i\neq \ell$ are equal to $\Id$, and then $F'=F$ as claimed.

It remains to show that $\pi'$ is 
an optimal bijection of~$S_{I^c,J^c}$ in~$G$. 
For this we remind that the optimality of $F$ is achieved by its weight, to which the  weights of $e_i$ (see Notation~\ref{N1})
do not contribute.
That is, by the rearrangement, the weight of $F$ is the weight of the rearrangement not including  the  weights of $e_i$,
which is the weight of the bijection  $\pi'$ (multiplied by weight $\unit$ of loops 
constituting the $k-1$ copies of identity).

By contradiction, assume that there exists a bijection  $\pi'':I^C\rightarrow J^C$ whose weight strictly surpasses the weight of $\pi'$. Then it is decomposed into  paths and cycles. It can be seen that the beginning nodes of the paths compose $I^C\setminus J^C= J\setminus I$ and the end nodes of the paths compose 
$J^C\setminus I^C=I\setminus J$. Completing each path into 
a cycle by a single edge and composing it with disjoint loops yields a permutation. Consider the set of such permutations and $|I\cap J|$ copies of $\Id$ corresponding to the nodes of $I\cap J$ (that 
can be neither beginning nor end nodes of paths in decomposition of $\pi''$). These permutations constitute 
a $(1,k)$-regular multigraph of $G$ with respect to $I,J$. The weight of this multigraph equals the weight of 
$\pi''$, which strictly surpasses the weight of $F$.

\textbf{Case~2:  Paths $P_i$ are not pairwise disjoint (see Figures~\ref{case2a}\ --\ \ref{case2c} and Notation~\ref{N1}\ --\ \ref{N3})}

In this case we will show that (2) occurs by proving that there exists $\tilde{\pi}\in S_{I,J}$ and $\tau_1,\dots,\tau_k\in S_n$ such that the multigraph
$F'=\big([n],\biguplus_{i\in[k]}E_{\tau_i},\Tilde{\pi}\big)$ is $(1,k)$-regular and optimal with respect to $I,J$, and
$$(\biguplus_{i\in[k]}E_{\tau_i})\setminus (E_{\tilde{\pi}}\cup L_{F'})\subseteq (\biguplus_{i\in[k]}E_{\rho_i})\setminus E_\pi.$$

Observe that Case~2 occurs when one of the conditions \ref{(a)}\ --\ \ref{(c)} in Case~1 fails. That is, if at least one of the following conditions holds:
\begin{enumerate}[\ \ \ (a)]
\item\label{(a')}  There exists a source which is also a target, or vice versa (of course the sources are disjoint and the targets are disjoint),

\item\label{(b')} There exists an intermediate node which is also a source or a target,

\item \label{(c')} There exists an intermediate node common to two paths.
\end{enumerate}

We will consider each of these cases separately.

\textbf{ Case~2\ref{(a')}:  There exists a source which is also a target (Figure~\ref{case2a})}

In this case $\exists i,j\in [k]:t(P_j)=s(P_i)$ .

We compose~$P_j\circ P_i$ and get a path~$Q$ such that$$s(Q)=s(P_j)\ ,\ t(Q)=t(P_i).$$
If $Q$ is an elementary path, then composed with the loops  on nodes disjoint to $V(E_Q)$, we get a bijection in $S_{\{t(P_i)\}^c,\{s(P_j)\}^c}$, which can be completed by one edge  into a permutation $\tau$. 
Therefore, taking $\tau_i=\tau\ ,\ \tau_j=\Id$, instead of $\rho_i,\rho_j$, and 
$\tau_\ell=\rho_\ell$ for all $\ell\ne i,j$, 
we obtain the $(1,k)$-regular multigraph 
  $F'=\big([n],\biguplus_{\ell\in[k]}E_{\tau_\ell},\tilde{\pi}\big)$ with respect to $I,J$  in~$G$, where 
$\tilde{\pi}$ is formed from $\pi$ by replacing the edges $(t(P_j),s(P_j))$ and $(t(P_i),s(P_i)$ 
with the edges $(t(P_i),s(P_j))$ and $(t(P_j),s(P_i))$.

Since the set of edges $ \bigcup_{i\in[k]}E_{\rho_i}\setminus E_\pi$ has not changed, 
but was rather rearranged, $F'$ is optimal:
$$\bigg(\sum_{i\in[k]}w(\rho_i)\bigg)-w( \pi)=
\bigg(\sum_{i\in[k]}w(\tau_i)\bigg)-w( \tilde{\pi}).$$
If $Q$ is not elementary, then it includes nontrivial cycles, and therefore its weight is less than or equal to the weight of
the elementary path $Q'$ from $s(Q)$ to $t(Q)$ included   in $Q$, where the cycles of $Q$ are  replaced by loops of nodes disjoint to $V(E_{Q'})$. As a result, the weight of $F$ is less than or equal to
the weight of the $(1,k)$-regular  multigraph  $F'$, where $\tau_i$ is now the elementary path $Q'$ composed with loops disjoint to $V(E_{Q'})$.  In the case when the weight of $F$ is strictly less, we have a contradiction with the optimality of $F$.  In the case when the weights are equal, we have found the desired $F'$.

\begin{center}
\begin{figure}[h]\begin{tikzpicture}
[main_node/.style={->,>=stealth',shorten >=-1pt,shorten <=-1pt,circle,fill=black,minimum size=0.8em,inner sep=1pt,line width=0.3mm},
prl_node/.style={->,>=stealth',shorten >=-1pt,shorten <=-1pt,circle,out=180,in=180,fill=black,minimum size=0.8em,inner sep=1pt,line width=0.5mm},
prr_node/.style={->,>=stealth',shorten >=-1pt,shorten <=-1pt,circle,out=0,in=0,fill=black,minimum size=0.8em,inner sep=1pt,line width=0.5mm},
arrow_node/.style={->,>=stealth',shorten >=-1pt,shorten <=-1pt,fill=black,minimum size=0.8em,inner sep=1pt,line width=0.5mm}]

\node(0) at (-4,-6.6) {$Q$};

\node[main_node](11) at (-7,-7) {};
\node (21) at (-7.6,-7) {$s(P_j)$};

\node[main_node](12) at (-4,-5) {};
\node (22) at (-4,-4.5) {$t(P_j)=s(P_i)$};

\node[main_node](13) at (-1,-7) {};
\node (23) at (-0.4,-7) {$t(P_i)$};

\draw[prr_node,red]   (11) edge node{} (12);
\draw[prl_node,blue]   (12) edge node{} (13);

\node[main_node](h) at (-4,-5.8) {};
\node (23) at (-4,-6.2) {};

\node (22) at (-4,-7.75) {\Large$\lleq$};

\end{tikzpicture}
\begin{tikzpicture}
[main_node/.style={->,>=stealth',shorten >=-1pt,shorten <=-1pt,circle,fill=black,minimum size=0.8em,inner sep=1pt,line width=0.3mm},
prl_node/.style={->,>=stealth',shorten >=-1pt,shorten <=-1pt,circle,out=-45,in=180,fill=black,minimum size=0.8em,inner sep=1pt,line width=0.5mm},
prr_node/.style={-,>=stealth',shorten >=-1pt,shorten <=-1pt,circle,out=0,in=-135,fill=black,minimum size=0.8em,inner sep=1pt,line width=0.5mm},
arrow_node/.style={->,>=stealth',shorten >=-1pt,shorten <=-1pt,fill=black,minimum size=0.8em,inner sep=1pt,line width=0.5mm}]

\node[main_node,white](0) at (-7,-5.5) {};
\node(0) at (-5.5,-6.5) {\Large$\tau$};

\node[main_node](31) at (-7,-7) {};
\node (21) at (-7.6,-7) {$s(P_j)$};

\node (22) at (-4,-7.5) {\Large$\text{composed with disjoint loops}$};

\node[main_node](h12) at (-4,-5.8) {};
\node (h2) at (-4,-6.2) {};

\node[main_node](33) at (-1,-7) {};
\node (23) at (-0.4,-7) {$t(P_i)$};

\draw[prr_node]   (31) edge node{} (h12);
\draw[prl_node]   (h12) edge node{} (33);
\end{tikzpicture}\caption{Case~(2a)}\label{case2a}\end{figure}
\end{center}

\textbf{Case~2\ref{(b')}:  There exists an intermediate node which is also a source or a target (see also Figure~\ref{case2b})}

This case occurs when
$$\exists i\in [k], t\in\overline{s(P_i)}\setminus\{s(P_i)\}:\ t\notin (I^c\cap J^c)\ (\Leftrightarrow t\in I\cup J).$$
Since $\pi\in S_{I,J}$, $\exists j\in [k]$ such that $t=t(P_j)\text{ or }t=s(P_j)\text{ (w.l.o.g. }t=s(P_j)).$  Hence we have, 

$$\exists i,j\in [k], t\in\overline{s(P_i)}:\ t=s(P_j).$$

We assume without loss of generality that Case~2a does not occur.

We compose~$P_i\circ P_j$,  and decompose $P_i$ at $t$, denoted by $Q_1\circ Q_2\circ P_j$   
such that $$s(Q_1)=s(P_i)\ ,\ t(Q_1)=t=s(Q_2)\ ,\ t(Q_2)=t(P_i).$$ We then write the composition as
$(Q_1\circ P_j)\circ Q_2$ where $Q_2$ is elementary, and~$Q_1\circ P_j$ is a path $Q$ such that $$s(Q)=s(P_i),t(Q)=t(P_j).$$ 
As before, if $Q$ is elementary, then composed with the loops  of nodes disjoint to $V(E_Q)$, we get a bijection in $S_{\{t(P_j)\}^c,\{s(P_i)\}^c}$, which can be completed by one edge into a permutation $\tau$. Composing $Q_2$  with the loops  of nodes disjoint to $V(E_{Q_2})$, we get a bijection in $S_{\{t(P_i)\}^c,\{s(P_j)\}^c}$, which can be completed by one edge  into a permutation $\tau'$. 
Therefore, taking $\tau_i=\tau\ ,\ \tau_j=\tau'$, instead of $\rho_i,\rho_j$, and 
$\tau_\ell=\rho_\ell$ for all $\ell\ne i,j$, 
we obtain the $(1,k)$-regular multigraph 
  $F'=\big([n],\biguplus_{\ell\in[k]}E_{\tau_\ell},\tilde{\pi}\big)$  with respect to $I,J$  in~$G$, where 
$\tilde{\pi}$ is formed from $\pi$ by replacing the edges $(t(P_j),s(P_j))$ and $(t(P_i),s(P_i)$ 
with the edges $(t(P_i),s(P_j))$ and $(t(P_j),s(P_i))$.
Since the set of edges $ \bigcup_{i\in[k]}E_{\rho_i}\setminus E_\pi$ has not changed, 
but was rather rearranged, $F'$ is optimal.

If $Q$ is not elementary, then it includes nontrivial cycles, and therefore its weight is  less than or equal to the weight of
the elementary path $Q'$ from $s(Q)$ to $t(Q)$ included   in $Q$, where the cycles of $Q$ are  replaced by loops of nodes disjoint to $V(E_{Q'})$. As a result, the weight of $F$ is less than or equal to
the weight of the $(1,k)$-regular   $F'$ above, where $\tau_i$ is the elementary path $Q'$ composed with loops disjoint to $V(E_{Q'})$.  As in case~2a, the weight of $F$ cannot be less than the weight of $F'$, so we have found the desired $F'$

\begin{center}
\begin{figure}[h]\begin{tikzpicture}
[main_node/.style={->,>=stealth',shorten >=-1pt,shorten <=-1pt,circle,fill=black,minimum size=0.8em,inner sep=1pt,line width=0.5mm},
prl_node/.style={->,>=stealth',shorten >=-1pt,shorten <=-1pt,circle,out=180,in=180,fill=black,minimum size=0.8em,inner sep=1pt,line width=0.5mm},
prr_node/.style={- ,>=stealth',shorten >=-1pt,shorten <=-1pt,circle,out=0,in=0,fill=black,minimum size=0.8em,inner sep=1pt,line width=0.5mm},
arrow_node/.style={->,>=stealth',shorten >=-1pt,shorten <=-1pt,fill=black,minimum size=0.8em,inner sep=1pt,line width=0.5mm}]

\node[main_node](11) at (-7,-7) {};
\node (21) at (-7,-7.5) {$s(P_i)$};

\node[main_node](14) at (-6,-5) {};
\node (24) at (-6.6,-5) {$t(P_i)$};

\node[main_node](12) at (-4,-5) {};
\node (22) at (-3.2,-4.65) {$t=s(P_j)$};

\node[main_node](13) at (-1,-7) {};
\node (23) at (-0.9,-7.5) {$t(P_j)$};

\draw[prr_node,red]   (11) edge node{} (12);
\draw[prl_node,blue]   (12) edge node{} (13);
\draw[main_node,red]   (12) edge node{} (14);
\node (21) at (-3.4,-5.6) {$\color{red}Q_1$};
\node (21) at (-3.4,-6.5) {$\color{blue}P_j$};
\node (21) at (-5,-4.8) {$\color{red}Q_2$};
\node (21) at (-4.6,-6.5) {$\color{red}Q_1$};

\node (eq) at (-0.35,-6) {\Large$=$};

\node[main_node](41) at (1,-7) {};
\node (51) at (1,-7.5) {$s(P_i)$};

\node[main_node](44) at (1.2,-5) {};
\node (54) at (0.7,-5) {$t(P_i)$};

\node[main_node](42) at (4,-5) {};
\node[main_node](421) at (3.2,-5) {};
\node (52) at (4.8,-4.65) {$t=s(P_j)$};

\node[main_node](43) at (7,-7) {};
\node (53) at (7,-7.5) {$t(P_j)$};

\draw[prr_node,red]   (41) edge node{} (42);
\draw[prl_node,blue]   (42) edge node{} (43);
\draw[main_node,red]   (421) edge node{} (44);

\node (51) at (2.2,-4.7) {$Q_2$};
\node (51) at (4,-6.5) {$Q$};

\node[main_node](h) at (-4,-5.8) {};
\node (23) at (-5,-6.2) {};

\node (22) at (4,-8) {\Large$\lleq$};

\node[main_node](117) at (4,-5.8) {};

\end{tikzpicture}
\begin{tikzpicture}
[main_node/.style={->,>=stealth',shorten >=-1pt,shorten <=-1pt,circle,fill=black,minimum size=0.8em,inner sep=1pt,line width=0.5mm},
prl_node/.style={->,>=stealth',shorten >=-1pt,circle,shorten <=-1pt,out=-45,in=180,fill=black,minimum size=0.8em,inner sep=1pt,line width=0.5mm},
prr_node/.style={- ,>=stealth',shorten >=-1pt,circle,shorten <=-1pt,out=0,in=-135,fill=black,minimum size=0.8em,inner sep=1pt,line width=0.5mm},
pr_node/.style={-> ,>=stealth',shorten >=-1pt,circle,shorten <=-1pt,out=0,in=-135,fill=black,minimum size=0.8em,inner sep=1pt,line width=0.5mm},
arrow_node/.style={->,>=stealth',shorten >=-1pt,shorten <=-1pt,fill=black,minimum size=0.8em,inner sep=1pt,line width=0.5mm}]
\node[main_node,color=white](11) at (-7.25,-7) {};

\node[main_node,white](0) at (1.5,-4.5) {};
\node(0) at (3,-6.6) {\Large$\tau$};
\node(0) at (3,-4.7) {\Large$\tau'$};

\node[main_node](31) at (1.5,-7) {};
\node (21) at (1.4,-7.5) {$s(P_i)$};

\node (22) at (4.5,-8) {\Large$\text{composed with  disjoint  loops}$};

\node[main_node](h12) at (4.5,-5.8) {};
\node (h2) at (4.5,-6.2) {};

\node[main_node](33) at (7.5,-7) {};
\node (23) at (7.6,-7.5) {$t(P_j)$};

\draw[prr_node]   (31) edge node{} (h12);
\draw[prl_node]   (h12) edge node{} (33);

\node[main_node](14) at (2,-5) {};
\node (24) at (1.5,-5) {$t(P_i)$};

\node[main_node](12) at (4,-5) {};
\node (22) at (5,-5) {$t=s(P_j)$};
\draw[main_node]   (12) edge node{} (14);

\end{tikzpicture}\caption{Case~(2b)}\label{case2b}\end{figure}
\end{center}

\textbf{Case~2\ref{(c')}:  There exists an intermediate node common to two paths (see Figure~\ref{case2c})}

In this case we have that,  $$\exists  t\in(\overline{s(P_i)}\setminus\{s(P_i)\})\cap(\overline{s(P_j)}
\setminus\{s(P_j)\})\text{ for some }i\ne j.$$

We assume without loss of generality that Cases 2a and 2b do not occur.

Let $P_i=Q_1\circ Q_2$ where $Q_1$ is the segment of $P_i$ between $s(P_i)$ and $t$, and $Q_2$ is the segment from $t$ to $t(P_i)$.

Similarly let $P_j=Q_3\circ Q_4$ where $s(Q_3)=s(P_j)$ ,$s(Q_2)=t=s(Q_4)$ and $t(Q_4)=t(P_j)$.

Then we can write the composition $P_i\circ P_j$ as $(Q_1\circ Q_4)\circ (Q_3\circ Q_2)$ where $$s(Q_1\circ Q_4)=s(P_i)\ ,\ s(Q_3\circ Q_2)=s(P_j)\ ,\ t(Q_1\circ Q_4)=t(P_j)\ ,\ t(Q_3\circ Q_2)=t(P_i).$$ 

Once again, if $Q_1\circ Q_4$ and $ Q_3\circ Q_2$ are elementary, then composed with the loops  of nodes disjoint to $V(E_{Q_1\circ Q_4})$ and $V(E_{Q_3\circ Q_2})$ respectively, we get   bijections in $S_{\{t(P_j)\}^c,\{s(P_i)\}^c}$ and $ S_{\{t(P_i)\}^c,\{s(P_j)\}^c}$, which can be completed, by one edge each, into permutations $\tau,\tau'$. 
Therefore, taking $\tau_i=\tau\ ,\ \tau_j=\tau'$, instead of $\rho_i,\rho_j$, and 
$\tau_\ell=\rho_\ell$ for all $\ell\ne i,j$, 
we obtain the $(1,k)$-regular multigraph 
  $F'=\big([n],\biguplus_{\ell\in[k]}E_{\tau_\ell},\tilde{\pi}\big)$  with respect to $I,J$  in~$G$, where 
$\tilde{\pi}$ is formed from $\pi$ by replacing the edges $(t(P_j),s(P_j))$ and $(t(P_i),s(P_i)$ 
with the edges $(t(P_i),s(P_j))$ and $(t(P_j),s(P_i))$.
Since the set of edges $ \bigcup_{i\in[k]}E_{\rho_i}\setminus E_\pi$ has not changed, 
but was rather rearranged, $F'$ is optimal.

\begin{center}
\begin{figure}[h]\begin{tikzpicture}
[main_node/.style={->,>=stealth',shorten >=-1pt,shorten <=-1pt,circle,fill=black,minimum size=0.8em,inner sep=1pt,line width=0.5mm},
prl_node/.style={- ,>=stealth',shorten >=-1pt,shorten <=-1pt,circle,out=0,in=0,fill=black,minimum size=0.8em,inner sep=1pt,line width=0.5mm},
pl_node/.style={->,>=stealth',shorten >=-1pt,shorten <=-1pt,circle,out=180,in=180,fill=black,minimum size=0.8em,inner sep=1pt,line width=0.5mm},
prr_node/.style={- ,>=stealth',shorten >=-1pt,shorten <=-1pt,circle,out=0,in=0,fill=black,minimum size=0.8em,inner sep=1pt,line width=0.5mm},
pr_node/.style={-> ,>=stealth',shorten >=-1pt,shorten <=-1pt,circle,out=180,in=180,fill=black,minimum size=0.8em,inner sep=1pt,line width=0.5mm},
arrow_node/.style={->,>=stealth',shorten >=-1pt,shorten <=-1pt,fill=black,minimum size=0.8em,inner sep=1pt,line width=0.5mm}]

\node[main_node](11) at (-7,-7) {};
\node (21) at (-7.6,-7) {$s(P_i)$};

\node[main_node](14) at (-6,-3) {};
\node (24) at (-6.6,-3) {$s(P_j)$};

\node[main_node](15) at (-2,-3) {};
\node (25) at (-1.3,-3) {$t(P_i)$};

\node[main_node](12) at (-4,-5) {};
\node (22) at (-4,-4.65) {$t$};

\node[main_node](13) at (-2,-7) {};
\node (23) at (-1.3,-7) {$t(P_j)$};

\draw[prr_node,red]   (11) edge node{} (12);
\draw[pl_node,blue]   (12) edge node{} (13);

\draw[pr_node,red]   (12) edge node{} (15);
\draw[prl_node,blue]   (14) edge node{} (12);

\node[main_node](h) at (-4,-4) {};
\node[main_node](h1) at (-4.1,-5.9) {};

\node (21) at (-4.6,-6.5) {$\color{red}Q_1$};
\node (21) at (-3.4,-5.6) {$\color{red}Q_1$};
\node (21) at (-4.7,-4.6) {$\color{red}Q_2$};
\node (21) at (-3.4,-3.2) {$\color{red}Q_2$};
\node (21) at (-3.3,-4.6) {$\color{blue}Q_3$};
\node (21) at (-4.6,-3.2) {$\color{blue}Q_3$};
\node (21) at (-3,-6.5) {$\color{blue}Q_4$};
\node (21) at (-4.6,-5.6) {$\color{blue}Q_4$};

\node (eq) at (-0.5,-5) {\Large$=$};

\node[main_node](41) at (1,-7) {};
\node (51) at (0.4,-7) {$s(P_i)$};

\node[main_node](44) at (2,-2.5) {};
\node (54) at (1.4,-2.5) {$s(P_j)$};

\node[main_node](45) at (6,-2.5) {};
\node (55) at (6.7,-2.5) {$t(P_i)$};

\node[main_node](42) at (4,-5) {};
\node (52) at (3.7,-4.75) {$t$};
\node[main_node](421) at (4,-4.5) {};

\node[main_node](43) at (6,-7) {};
\node (53) at (6.7,-7) {$t(P_j)$};

\draw[prr_node,red]   (41) edge node{} (42);
\draw[pl_node,blue]   (42) edge node{} (43);

\draw[pr_node,red]   (421) edge node{} (45);
\draw[prl_node,blue]   (44) edge node{} (421);

\node[main_node](t) at (4,-3.5) {};
\node[main_node](t1) at (3.9,-5.9) {};

\node (52) at (4,-7.75) {\Large$\lleq$};

\node (51) at (3.75,-6.75) {$Q_1\circ Q_4$};

\node (51) at (4,-2.8) {$Q_3\circ Q_2$};

\end{tikzpicture}
\begin{tikzpicture}
[main_node/.style={->,>=stealth',shorten >=-1pt,shorten <=-1pt,circle,fill=black,minimum size=0.8em,inner sep=1pt,line width=0.5mm},
pl_node/.style={->,>=stealth',shorten >=-1pt,circle,shorten <=-1pt,out=-45,in=180,fill=black,minimum size=0.8em,inner sep=1pt,line width=0.5mm},
prl_node/.style={-,>=stealth',shorten >=-1pt,circle,shorten <=-1pt,out=135,in=0,fill=black,minimum size=0.8em,inner sep=1pt,line width=0.5mm},
prr_node/.style={- ,>=stealth',shorten >=-1pt,circle,shorten <=-1pt,out=0,in=-135,fill=black,minimum size=0.8em,inner sep=1pt,line width=0.5mm},
pr_node/.style={-> ,>=stealth',shorten >=-1pt,circle,shorten <=-1pt,out=45,in=180,fill=black,minimum size=0.8em,inner sep=1pt,line width=0.5mm},
arrow_node/.style={->,>=stealth',shorten >=-1pt,shorten <=-1pt,fill=black,minimum size=0.8em,inner sep=1pt,line width=0.5mm}]
\node[main_node,color=white](11) at (-6,-7) {};

\node[main_node,white](0) at (2,-4.5) {};
\node(0) at (3.5,-6.6) {\Large$\tau$};
\node(0) at (4.9,-3.5) {\Large$\tau'$};

\node[main_node](31) at (2,-7) {};
\node (21) at (1.4,-7) {$s(P_i)$};

\node (22) at (4.5,-4.9) {\Large$\text{composed with disjoint loops}$};

\node[main_node](h12) at (4.85,-5.8) {};
\node[main_node](h) at (5,-4) {};

\node[main_node](33) at (7,-7) {};
\node (23) at (7.7,-7) {$t(P_j)$};

\node[main_node](14) at (3,-3) {};
\node (24) at (2.4,-3) {$s(P_j)$};

\node[main_node](15) at (7,-3) {};
\node (25) at (7.7,-3) {$t(P_i)$};

\draw[prr_node]   (31) edge node{} (h12);
\draw[pl_node]   (h12) edge node{} (33);
\draw[pr_node]   (h) edge node{} (15);
\draw[prl_node]   (h) edge node{} (14);

\end{tikzpicture}\caption{Case~(2c)}\label{case2c}\end{figure}
\end{center}

If $Q_1\circ Q_4$ (resp.~$Q_3\circ Q_2$) is not elementary, then it includes nontrivial cycles, and therefore its weight is  less than or equal to
the weight of the elementary path $Q'$ from $s(Q_1\circ Q_4)$ (resp.~$s(Q_3\circ Q_2)$) to $t(Q_1\circ Q_4)$ (resp.~$t(Q_3\circ Q_2)$) included   in $Q_1\circ Q_4$ (resp.~$Q_3\circ Q_2$), 
where the cycles of  $Q_1\circ Q_4$ (resp.~$Q_3\circ Q_2$) are  replaced by loops of nodes disjoint to $V(E_{Q'})$. As a result, the weight of $F$ is less than or equal to
the weight of the $(1,k)$-regular   $F'$ above, where $\tau_i$ (resp.~$\tau_j$) is the elementary path $Q'$ composed with loops disjoint to $V(E_{Q'})$.  We conclude in the same way as in Cases 2a and 2b.

Observe that in all subcases of Case~2 we change the set of edges of two permutations in the multigraph,
and therefore $F'\neq F$.

  \end{proof}

\begin{rem}
If, in addition, $\Id$ is the unique optimal permutation then 
every optimal multigraph $F$ has the required property by Lemma~\ref{norm}.
\end{rem}

Let us also formulate a version of the above theorem which applies to any optimal multigraph.

\begin{cor}\label{cor:JacobiGraph}  Let $G=([n],E)$ be a simple weighted digraph
where $\Id$ is an optimal permutation.  
Let $F=\big([n],\biguplus_{i\in[k]}E_{\rho_i},\pi\big)$ be an arbitrary optimal~$(1,k)$-regular multigraph  of~$G$ with respect to 
$I,J\subseteq [n]$.
Then there exists a $k$-regular multigraph~$F'=\big([n],\biguplus_{i\in[k]}E_{\tau_i}\big)$ of~$G$ such that 
at least one of the following statements holds 
\begin{enumerate} 
\item  
$E_\pi\subseteq E_{\tau_\ell}$ for some $\ell\in [k]$, and 
$\pi'\in S_{I^c,J^c}$ defined by $E_{\pi'}=E_{\tau_\ell}\setminus E_\pi $
satisfies $\bigcup_{i\in[k]} E_{P_i}\subseteq E_{\pi'}$ and  
is an optimal bijection (with respect to $I^c, J^c$). 

\item There exists $\tilde{\pi}\in S_{I,J}$ such that
 $F'=\big([n],\biguplus_{i\in[k]}E_{\tau_i},\tilde{\pi}\big)\neq F$ satisfies $(\biguplus_{i\in[k]}E_{\tau_i})\setminus (E_{\tilde{\pi}}\cup L_{F'})\subseteq (\biguplus_{i\in[k]}E_{\rho_i})\setminus E_\pi$ and is also an optimal $(1,k)$-regular multigraph with respect to $I,J$.
\end{enumerate} 
\end{cor}

\begin{proof}
Using Lemma~\ref{norm}, multigraph $F$ can be replaced with a multigraph with $\rho_i$
having the same cycles $C_i$ as in $F$ and with all other cycles being loops, maintaining the optimality. 
Then all properties of both alternatives follow from the corresponding alternatives in
Theorem~\ref{thm:JacobiGraph}.

\end{proof}

\begin{rem}\label{nnrm} (On the General Case)
Identity is not an optimal permutation for
a general matrix, but any optimal permutation can be relocated to the
diagonal by permuting columns (or rows). 
Let us consider this process on the digraph $G_A$  corresponding to $A\in\rmax^{n\times n}$, 
and furthermore on the 
bipartite graph corresponding to $G_A$. Note that the permutations are interpreted as assignments
matching $n$ workers to $n$ tasks. 
In that sense, one has the freedom of renumbering the tasks. 
If assigning worker $i$ to task $\pi(i)$ for every $i\in[n]$ is
 an optimal assignment, then one can count task $\pi(i)$ as task $i$.
This is easy to do once an optimal permutation is known, hence the condition that $\Id$ should be 
optimal is not restrictive in practice.
\end{rem}

\if{
\begin{rem}\label{nnrm}(On the General Case) 
Recalling Lemma~\ref{norm}(2), any matrix 
$A\in \R_{\max}^{n\times n}$ whose digraph contains at least one permutation with a finite weight can be normalized by relocating its optimal permutation $\pi$ to the diagonal and normalizing its
 entries $A_{i,\pi(i)}$ to $\unit$ (just subtracting $A_{i,\pi(i)}$ from each entry of the 
column in the sense of max-plus arithmetic). 

Let us consider this process on the digraph $G_A$  corresponding to $A$, and further more on the 
bipartite graph corresponding to $G_A$. Note that the permutations are interpreted as assignments
 matching $n$ workers to 
$n$ tasks. 
In that sense, one has the freedom of renumbering the tasks. 
If assigning worker $i$ to task $\pi(i)$ for every $i\in[n]$ is
 an optimal assignment, then one can count task $\pi(i)$ as task $i$. The value normalization is a well known formality, which is 
often  preformed in grading, prioritizing and scoring processes. That is, the weight optimality  is taken for the overall assignment, 
and the edge weights are only considered as part of assignments. Therefore the value of $i$ performing task $j$ can be normalized according to $i$ preforming task $\pi(i)$ (or job $i$)..

  We conclude that one may adjust $G$ in Theorem~\ref{thm:JacobiGraph} to be free of the restrictions in Notation~\ref{thmnota}.
 We claim  however, that for the purposes of the graph-theoretic Jacobi identity, the optimal permutation is essentially the identity, 
and the assignment values should be straighten  according to the desired one.
\end{rem}
}\fi

\subsection{Example showing both cases of tropical Jacobi identity in terms of assignments and supervisions}

Here we consider how a set of $k$ assignments with supervisions $I$ on $J$ links to the cases observed in Theorems \ref{thm:JacobiIdentity}and \ref{thm:JacobiGraph}.

\newsavebox{\adjAtwotwo}
\begin{lrbox}{\adjAtwotwo}
\begin {tikzpicture}[node distance =0.2cm and 0.5cm ,on grid ,semithick, Vertex/.style ={circle, inner sep=1pt, fill}]
\node[Vertex] (A1){};\node[Vertex] (A2) [below =of A1] {};\node[Vertex] (A3) [below =of A2] {};\node[Vertex] (A4)[below =of A3] {};
\node[Vertex] (B1)[right =of A1] {};\node[Vertex] (B2) [below =of B1] {};\node[Vertex] (B3) [below =of B2] {};\node[Vertex] (B4)[below =of B3] {};
\draw (A1) to (B1)   (A3) to (B3) (A4) to (B4);\end{tikzpicture}\end{lrbox}
\newsavebox{\adjAthreethree}
\begin{lrbox}{\adjAthreethree}
\begin {tikzpicture}[node distance =0.2cm and 0.5cm ,on grid ,semithick, Vertex/.style ={circle, inner sep=1pt, fill}]
\node[Vertex] (A1){};\node[Vertex] (A2) [below =of A1] {};\node[Vertex] (A3) [below =of A2] {};\node[Vertex] (A4)[below =of A3] {};
\node[Vertex] (B1)[right =of A1] {};\node[Vertex] (B2) [below =of B1] {};\node[Vertex] (B3) [below =of B2] {};\node[Vertex] (B4)[below =of B3] {};
\draw (A1) to (B1) (A2) to (B2) (A4) to (B4);\end{tikzpicture}\end{lrbox} 
\newsavebox{\adjAfourone}
\begin{lrbox}{\adjAfourone}
\begin {tikzpicture}[node distance =0.2cm and 0.5cm ,on grid ,semithick, Vertex/.style ={circle, inner sep=1pt, fill}]
\node[Vertex] (A1){};\node[Vertex] (A2) [below =of A1] {};\node[Vertex] (A3) [below =of A2] {};\node[Vertex] (A4)[below =of A3] {};
\node[Vertex] (B1)[right =of A1] {};\node[Vertex] (B2) [below =of B1] {};\node[Vertex] (B3) [below =of B2] {};\node[Vertex] (B4)[below =of B3] {};
\draw  (A2) to (B2)  (A3) to (B3) (A4) to (B1);\end{tikzpicture}\end{lrbox}

\begin{exa} $$A=\begin{pmatrix}0&-1&-5&-4\\-6&0&-2&-1\\-3&-4&0&-3\\-2&-7&0&0\end{pmatrix}.$$  We calculate $\adj(A)$ and also show the table of bijections (in $A$) corresponding to each entry of the adjoint. 
$$\adj(A)=\begin{pmatrix}0&-1&-2&-2\\-3&0&-1&-1\\-3&-4^{\bullet}&0&-3\\-2&-3&0&0\end{pmatrix} \leftrightarrow
\left(\begin{array}{c c c c}
\begin {tikzpicture}[node distance =0.2cm and 0.5cm ,on grid ,semithick, Vertex/.style ={circle, inner sep=1pt, fill}]
\node[Vertex] (A1){};\node[Vertex] (A2) [below =of A1] {};\node[Vertex] (A3) [below =of A2] {};\node[Vertex] (A4)[below =of A3] {};
\node[Vertex] (B1)[right =of A1] {};\node[Vertex] (B2) [below =of B1] {};\node[Vertex] (B3) [below =of B2] {};\node[Vertex] (B4)[below =of B3] {};
\draw (A2) to (B2)  (A3) to (B3) (A4) to (B4);\end{tikzpicture}
&
\begin {tikzpicture}[node distance =0.2cm and 0.5cm ,on grid ,semithick, Vertex/.style ={circle, inner sep=1pt, fill}]
\node[Vertex] (A1){};\node[Vertex] (A2) [below =of A1] {};\node[Vertex] (A3) [below =of A2] {};\node[Vertex] (A4)[below =of A3] {};
\node[Vertex] (B1)[right =of A1] {};\node[Vertex] (B2) [below =of B1] {};\node[Vertex] (B3) [below =of B2] {};\node[Vertex] (B4)[below =of B3] {};
\draw (A1) to (B2) (A3) to (B3) (A4) to (B4);\end{tikzpicture}
&
\begin {tikzpicture}[node distance =0.2cm and 0.5cm ,on grid ,semithick, Vertex/.style ={circle, inner sep=1pt, fill}]
\node[Vertex] (A1){};\node[Vertex] (A2) [below =of A1] {};\node[Vertex] (A3) [below =of A2] {};\node[Vertex] (A4)[below =of A3] {};
\node[Vertex] (B1)[right =of A1] {};\node[Vertex] (B2) [below =of B1] {};\node[Vertex] (B3) [below =of B2] {};\node[Vertex] (B4)[below =of B3] {};
\draw (A1) to (B2) (A2) to (B4)  (A4) to (B3);\end{tikzpicture}
&
\begin {tikzpicture}[node distance =0.2cm and 0.5cm ,on grid ,semithick, Vertex/.style ={circle, inner sep=1pt, fill}]
\node[Vertex] (A1){};\node[Vertex] (A2) [below =of A1] {};\node[Vertex] (A3) [below =of A2] {};\node[Vertex] (A4)[below =of A3] {};
\node[Vertex] (B1)[right =of A1] {};\node[Vertex] (B2) [below =of B1] {};\node[Vertex] (B3) [below =of B2] {};\node[Vertex] (B4)[below =of B3] {};
\draw (A1) to (B2) (A2) to (B4)  (A3) to (B3);\end{tikzpicture}
\\ 
\begin {tikzpicture}[node distance =0.2cm and 0.5cm ,on grid ,semithick, Vertex/.style ={circle, inner sep=1pt, fill}]
\node[Vertex] (A1){};\node[Vertex] (A2) [below =of A1] {};\node[Vertex] (A3) [below =of A2] {};\node[Vertex] (A4)[below =of A3] {};
\node[Vertex] (B1)[right =of A1] {};\node[Vertex] (B2) [below =of B1] {};\node[Vertex] (B3) [below =of B2] {};\node[Vertex] (B4)[below =of B3] {};
\draw  (A2) to (B4)  (A3) to (B3) (A4) to (B1);\end{tikzpicture}
&
\usebox{\adjAtwotwo}
&
\begin {tikzpicture}[node distance =0.2cm and 0.5cm ,on grid ,semithick, Vertex/.style ={circle, inner sep=1pt, fill}]
\node[Vertex] (A1){};\node[Vertex] (A2) [below =of A1] {};\node[Vertex] (A3) [below =of A2] {};\node[Vertex] (A4)[below =of A3] {};
\node[Vertex] (B1)[right =of A1] {};\node[Vertex] (B2) [below =of B1] {};\node[Vertex] (B3) [below =of B2] {};\node[Vertex] (B4)[below =of B3] {};
\draw (A1) to (B1) (A2) to (B4) (A4) to (B3);\end{tikzpicture}
&
\begin {tikzpicture}[node distance =0.2cm and 0.5cm ,on grid ,semithick, Vertex/.style ={circle, inner sep=1pt, fill}]
\node[Vertex] (A1){};\node[Vertex] (A2) [below =of A1] {};\node[Vertex] (A3) [below =of A2] {};\node[Vertex] (A4)[below =of A3] {};
\node[Vertex] (B1)[right =of A1] {};\node[Vertex] (B2) [below =of B1] {};\node[Vertex] (B3) [below =of B2] {};\node[Vertex] (B4)[below =of B3] {};
\draw (A1) to (B1) (A2) to (B4)  (A3) to (B3);\end{tikzpicture}
\\ 
\begin {tikzpicture}[node distance =0.2cm and 0.5cm ,on grid ,semithick, Vertex/.style ={circle, inner sep=1pt, fill}]
\node[Vertex] (A1){};\node[Vertex] (A2) [below =of A1] {};\node[Vertex] (A3) [below =of A2] {};\node[Vertex] (A4)[below =of A3] {};
\node[Vertex] (B1)[right =of A1] {};\node[Vertex] (B2) [below =of B1] {};\node[Vertex] (B3) [below =of B2] {};\node[Vertex] (B4)[below =of B3] {};
\draw  (A2) to (B2)  (A3) to (B1) (A4) to (B4);\end{tikzpicture}
&
\begin {tikzpicture}[node distance =0.2cm and 0.5cm ,on grid ,semithick, Vertex/.style ={circle, inner sep=1pt, fill}]
\node[Vertex] (A1){};\node[Vertex] (A2) [below =of A1] {};\node[Vertex] (A3) [below =of A2] {};\node[Vertex] (A4)[below =of A3] {};
\node[Vertex] (B1)[right =of A1] {};\node[Vertex] (B2) [below =of B1] {};\node[Vertex] (B3) [below =of B2] {};\node[Vertex] (B4)[below =of B3] {};
\draw (A1) to (B2)   (A3) to (B1) (A4) to (B4);\end{tikzpicture}\text{ or }
\begin {tikzpicture}[node distance =0.2cm and 0.5cm ,on grid ,semithick, Vertex/.style ={circle, inner sep=1pt, fill}]
\node[Vertex] (A1){};\node[Vertex] (A2) [below =of A1] {};\node[Vertex] (A3) [below =of A2] {};\node[Vertex] (A4)[below =of A3] {};
\node[Vertex] (B1)[right =of A1] {};\node[Vertex] (B2) [below =of B1] {};\node[Vertex] (B3) [below =of B2] {};\node[Vertex] (B4)[below =of B3] {};
\draw (A1) to (B1)   (A3) to (B2) (A4) to (B4);\end{tikzpicture}
&
\usebox{\adjAthreethree}
&
\begin {tikzpicture}[node distance =0.2cm and 0.5cm ,on grid ,semithick, Vertex/.style ={circle, inner sep=1pt, fill}]
\node[Vertex] (A1){};\node[Vertex] (A2) [below =of A1] {};\node[Vertex] (A3) [below =of A2] {};\node[Vertex] (A4)[below =of A3] {};
\node[Vertex] (B1)[right =of A1] {};\node[Vertex] (B2) [below =of B1] {};\node[Vertex] (B3) [below =of B2] {};\node[Vertex] (B4)[below =of B3] {};
\draw (A1) to (B1) (A2) to (B2)  (A3) to (B4);\end{tikzpicture}
\\  
\usebox{\adjAfourone}
&
\begin {tikzpicture}[node distance =0.2cm and 0.5cm ,on grid ,semithick, Vertex/.style ={circle, inner sep=1pt, fill}]
\node[Vertex] (A1){};\node[Vertex] (A2) [below =of A1] {};\node[Vertex] (A3) [below =of A2] {};\node[Vertex] (A4)[below =of A3] {};
\node[Vertex] (B1)[right =of A1] {};\node[Vertex] (B2) [below =of B1] {};\node[Vertex] (B3) [below =of B2] {};\node[Vertex] (B4)[below =of B3] {};
\draw (A1) to (B2)   (A3) to (B3) (A4) to (B1);\end{tikzpicture}
&
\begin {tikzpicture}[node distance =0.2cm and 0.5cm ,on grid ,semithick, Vertex/.style ={circle, inner sep=1pt, fill}]
\node[Vertex] (A1){};\node[Vertex] (A2) [below =of A1] {};\node[Vertex] (A3) [below =of A2] {};\node[Vertex] (A4)[below =of A3] {};
\node[Vertex] (B1)[right =of A1] {};\node[Vertex] (B2) [below =of B1] {};\node[Vertex] (B3) [below =of B2] {};\node[Vertex] (B4)[below =of B3] {};
\draw (A1) to (B1) (A2) to (B2) (A4) to (B3);\end{tikzpicture}
&
\begin {tikzpicture}[node distance =0.2cm and 0.5cm ,on grid ,semithick, Vertex/.style ={circle, inner sep=1pt, fill}]
\node[Vertex] (A1){};\node[Vertex] (A2) [below =of A1] {};\node[Vertex] (A3) [below =of A2] {};\node[Vertex] (A4)[below =of A3] {};
\node[Vertex] (B1)[right =of A1] {};\node[Vertex] (B2) [below =of B1] {};\node[Vertex] (B3) [below =of B2] {};\node[Vertex] (B4)[below =of B3] {};
\draw (A1) to (B1) (A2) to (B2)  (A3) to (B3);\end{tikzpicture}
\end{array}\right).$$

\underline{Theorem \ref{thm:JacobiGraph} Case~1}  $\adj(A)^{\wedge 3}_{\{2,3,4\}, \{1,2,3\}}=-2=a_{41}=A^{\wedge 1}_{\{4\}\{1\}}$ thus the tropical Jacobi identity holds with equality for this choice of $I$ and $J$.  

In this case there is only one set  of $3$ assignments with supervisions $I$ on $J$ which achieves the maximum base value.
This value $\adj(A)^{\wedge 3}_{\{2,3,4\}, \{1,2,3\}}$ is attained for the bijection (in $\adj(A)$) $(2,2),(3,3),(4,1)$ which corresponds to the following three bijections (in $A$), that is, three assignments with supervisions:  

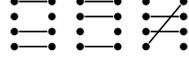
\begin{center}\usebox{\adjAtwotwo} \ \usebox{\adjAthreethree} \ \usebox{\adjAfourone}\end{center}

We can rearrange the solid edges so that we have $k-1=2$ full permutations in $S_4$, and a bijection representing the optimal permutation of $A^{\wedge n-k}$ which can be completed to a full permutation by adding the edges representing supervisions.  Indeed, in the figure below we have, on the left, the 3 assignments with supervisions described by the bijections.  On the right we rearrange the edges (note that they both represent the same digraph, and have the same value).

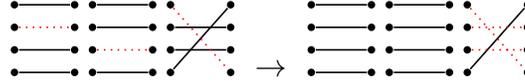
\begin{center}
\begin {tikzpicture}[node distance =0.3cm and 0.8cm ,on grid ,semithick, Vertex/.style ={circle, inner sep=1pt, fill}]
\node[Vertex] (A1){};\node[Vertex] (A2) [below =of A1] {};\node[Vertex] (A3) [below =of A2] {};\node[Vertex] (A4)[below =of A3] {};
\node[Vertex] (B1)[right =of A1] {};\node[Vertex] (B2) [below =of B1] {};\node[Vertex] (B3) [below =of B2] {};\node[Vertex] (B4)[below =of B3] {};
\draw (A1) to (B1)   (A3) to (B3) (A4) to (B4);
\draw[red, dotted] (A2) to (B2);\end{tikzpicture}
\begin {tikzpicture}[node distance =0.3cm and 0.8cm ,on grid ,semithick, Vertex/.style ={circle, inner sep=1pt, fill}]
\node[Vertex] (A1){};\node[Vertex] (A2) [below =of A1] {};\node[Vertex] (A3) [below =of A2] {};\node[Vertex] (A4)[below =of A3] {};
\node[Vertex] (B1)[right =of A1] {};\node[Vertex] (B2) [below =of B1] {};\node[Vertex] (B3) [below =of B2] {};\node[Vertex] (B4)[below =of B3] {};
\draw (A1) to (B1) (A2) to (B2) (A4) to (B4);
\draw[red, dotted] (A3) to (B3);\end{tikzpicture}
\begin {tikzpicture}[node distance =0.3cm and 0.8cm ,on grid ,semithick, Vertex/.style ={circle, inner sep=1pt, fill}]
\node[Vertex] (A1){};\node[Vertex] (A2) [below =of A1] {};\node[Vertex] (A3) [below =of A2] {};\node[Vertex] (A4)[below =of A3] {};
\node[Vertex] (B1)[right =of A1] {};\node[Vertex] (B2) [below =of B1] {};\node[Vertex] (B3) [below =of B2] {};\node[Vertex] (B4)[below =of B3] {};
\draw  (A2) to (B2)  (A3) to (B3) (A4) to (B1);
\draw[red, dotted] (A1) to (B4);\end{tikzpicture}
\ $\rightarrow$ \ 
\begin {tikzpicture}[node distance =0.3cm and 0.8cm ,on grid ,semithick, Vertex/.style ={circle, inner sep=1pt, fill}]
\node[Vertex] (A1){};\node[Vertex] (A2) [below =of A1] {};\node[Vertex] (A3) [below =of A2] {};\node[Vertex] (A4)[below =of A3] {};
\node[Vertex] (B1)[right =of A1] {};\node[Vertex] (B2) [below =of B1] {};\node[Vertex] (B3) [below =of B2] {};\node[Vertex] (B4)[below =of B3] {};
\draw (A1) to (B1)   (A3) to (B3) (A4) to (B4) (A2) to (B2);\end{tikzpicture}
\begin {tikzpicture}[node distance =0.3cm and 0.8cm ,on grid ,semithick, Vertex/.style ={circle, inner sep=1pt, fill}]
\node[Vertex] (A1){};\node[Vertex] (A2) [below =of A1] {};\node[Vertex] (A3) [below =of A2] {};\node[Vertex] (A4)[below =of A3] {};
\node[Vertex] (B1)[right =of A1] {};\node[Vertex] (B2) [below =of B1] {};\node[Vertex] (B3) [below =of B2] {};\node[Vertex] (B4)[below =of B3] {};
\draw (A1) to (B1)   (A3) to (B3) (A4) to (B4) (A2) to (B2);\end{tikzpicture}
\begin {tikzpicture}[node distance =0.3cm and 0.8cm ,on grid ,semithick, Vertex/.style ={circle, inner sep=1pt, fill}]
\node[Vertex] (A1){};\node[Vertex] (A2) [below =of A1] {};\node[Vertex] (A3) [below =of A2] {};\node[Vertex] (A4)[below =of A3] {};
\node[Vertex] (B1)[right =of A1] {};\node[Vertex] (B2) [below =of B1] {};\node[Vertex] (B3) [below =of B2] {};\node[Vertex] (B4)[below =of B3] {};
\draw (A4) to (B1);
\draw[red, dotted] (A1) to (B4) (A2) to (B2)  (A3) to (B3) ;\end{tikzpicture}\end{center}

By rearranging the edges in this way, we have found that the base value of the $k$ assignments (the sum of the weights of the solid assignments on the left hand side) is equal to the sum of the value of two identity permutations (which achieve the value of $per(A)$) and one bijection from $J^C=\{4\}$ to $I^C=\{1\}$ (which reflects the value of $A^{\wedge (n-k)}_{J^C, I^C}$). 
Further, in this case, the supervised edges are always constitute a bijection $\sigma$ from $J$ to $I$  which is the complement to the bijection $\tau$ from $J^C$ to $I^C$ which attains the value of $A^{\wedge (n-k)}_{J^C, I^C}$ in the sense that together $\sigma$ and $\tau$ form a permutation in $S_n$.

\underline{ Case~2a in the proof of Theorem~\ref{thm:JacobiGraph}}  For $I=\{1,2,3 \}$ and $J=\{1,3,4 \}$ we have $$\adj(A)^{\wedge 3}_{J,I}=-3^{\bullet} >A^{\wedge 1}_{I^C,J^C}=A_{4,2}=-7.$$ This is attained by two bijections (in $\adj(A)$), $(1,2),(3,3),(4,1)$ and $(1,1),(3,3),(4,2),$ These represent, in $A$, the following choices for 3 assignments with supervisions.  Note that the supervisions change, but the edges representing the assignments are only rearranged.

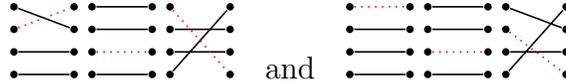
\begin{center}
\begin {tikzpicture}[node distance =0.3cm and 0.8cm ,on grid ,semithick, Vertex/.style ={circle, inner sep=1pt, fill}]
\node[Vertex] (A1){};\node[Vertex] (A2) [below =of A1] {};\node[Vertex] (A3) [below =of A2] {};\node[Vertex] (A4)[below =of A3] {};
\node[Vertex] (B1)[right =of A1] {};\node[Vertex] (B2) [below =of B1] {};\node[Vertex] (B3) [below =of B2] {};\node[Vertex] (B4)[below =of B3] {};
\draw  (A1) to (B2)  (A3) to (B3) (A4) to (B4);
\draw[red, dotted] (A2) to (B1);\end{tikzpicture}
\begin {tikzpicture}[node distance =0.3cm and 0.8cm ,on grid ,semithick, Vertex/.style ={circle, inner sep=1pt, fill}]
\node[Vertex] (A1){};\node[Vertex] (A2) [below =of A1] {};\node[Vertex] (A3) [below =of A2] {};\node[Vertex] (A4)[below =of A3] {};
\node[Vertex] (B1)[right =of A1] {};\node[Vertex] (B2) [below =of B1] {};\node[Vertex] (B3) [below =of B2] {};\node[Vertex] (B4)[below =of B3] {};
\draw (A1) to (B1) (A2) to (B2)  (A4) to (B4);
\draw[red, dotted] (A3) to (B3);\end{tikzpicture}
\begin {tikzpicture}[node distance =0.3cm and 0.8cm ,on grid ,semithick, Vertex/.style ={circle, inner sep=1pt, fill}]
\node[Vertex] (A1){};\node[Vertex] (A2) [below =of A1] {};\node[Vertex] (A3) [below =of A2] {};\node[Vertex] (A4)[below =of A3] {};
\node[Vertex] (B1)[right =of A1] {};\node[Vertex] (B2) [below =of B1] {};\node[Vertex] (B3) [below =of B2] {};\node[Vertex] (B4)[below =of B3] {};
\draw (A4) to (B1) (A2) to (B2)  (A3) to (B3);
\draw[red, dotted] (A1) to (B4);\end{tikzpicture}
\ \text{ and } \ 
\begin {tikzpicture}[node distance =0.3cm and 0.8cm ,on grid ,semithick, Vertex/.style ={circle, inner sep=1pt, fill}]
\node[Vertex] (A1){};\node[Vertex] (A2) [below =of A1] {};\node[Vertex] (A3) [below =of A2] {};\node[Vertex] (A4)[below =of A3] {};
\node[Vertex] (B1)[right =of A1] {};\node[Vertex] (B2) [below =of B1] {};\node[Vertex] (B3) [below =of B2] {};\node[Vertex] (B4)[below =of B3] {};
\draw (A2) to (B2)  (A3) to (B3) (A4) to (B4);
\draw[red, dotted] (A1) to (B1);\end{tikzpicture}
\begin {tikzpicture}[node distance =0.3cm and 0.8cm ,on grid ,semithick, Vertex/.style ={circle, inner sep=1pt, fill}]
\node[Vertex] (A1){};\node[Vertex] (A2) [below =of A1] {};\node[Vertex] (A3) [below =of A2] {};\node[Vertex] (A4)[below =of A3] {};
\node[Vertex] (B1)[right =of A1] {};\node[Vertex] (B2) [below =of B1] {};\node[Vertex] (B3) [below =of B2] {};\node[Vertex] (B4)[below =of B3] {};
\draw (A1) to (B1) (A2) to (B2)  (A4) to (B4);
\draw[red, dotted] (A3) to (B3);\end{tikzpicture}
\begin {tikzpicture}[node distance =0.3cm and 0.8cm ,on grid ,semithick, Vertex/.style ={circle, inner sep=1pt, fill}]
\node[Vertex] (A1){};\node[Vertex] (A2) [below =of A1] {};\node[Vertex] (A3) [below =of A2] {};\node[Vertex] (A4)[below =of A3] {};
\node[Vertex] (B1)[right =of A1] {};\node[Vertex] (B2) [below =of B1] {};\node[Vertex] (B3) [below =of B2] {};\node[Vertex] (B4)[below =of B3] {};
\draw (A1) to (B2) (A4) to (B1)  (A3) to (B3);
\draw[red, dotted] (A2) to (B4);\end{tikzpicture}\end{center}

This is what was expected from Theorem \ref{thm:JacobiGraph};  that the daily assignments of workers to jobs which are not a supervised assignment (the solid edges in the above) can be swapped between days to give another set of $k$ assignments with the same base value but a different set of supervisions $I$ on $J$.

\underline{Case~2b in the proof of Theorem~\ref{thm:JacobiGraph}}
Finally, it can be verified that, for $I=\{1,2\}$ and $J=\{3,4\}$ we have
 $$\adj(A)^{\wedge 2}_{J, I}=-6^{\bullet}=(\adj(A)_{3,1}\adj(A)_{4,2})\oplus(\adj(A)_{3,2}\adj(A)_{4,1})$$ and $A^{\wedge 2}_{I^C, J^C}=-6=A_{3,2}A_{4,1}.$  In this case  
  $\adj(A)^{\wedge 2}_{J, I}$ is attained twice and equality holds in the tropical Jacobi identity.  There are three sets of 2 assignments obtaining the optimal base value in this case, shown below.  The final one can be rearranged into an identity permutation, and the edges $(3,2)$ and $(4,1)$ which reflect the bijection attaining the value of $A^{\wedge 2}_{I^C, J^C}$.

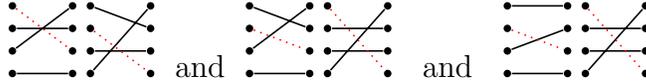
\begin{center}
\begin {tikzpicture}[node distance =0.3cm and 0.8cm ,on grid ,semithick, Vertex/.style ={circle, inner sep=1pt, fill}]
\node[Vertex] (A1){};\node[Vertex] (A2) [below =of A1] {};\node[Vertex] (A3) [below =of A2] {};\node[Vertex] (A4)[below =of A3] {};
\node[Vertex] (B1)[right =of A1] {};\node[Vertex] (B2) [below =of B1] {};\node[Vertex] (B3) [below =of B2] {};\node[Vertex] (B4)[below =of B3] {};
\draw  (A2) to (B2)  (A3) to (B1) (A4) to (B4);
\draw[red, dotted] (A1) to (B3);\end{tikzpicture}
\begin {tikzpicture}[node distance =0.3cm and 0.8cm ,on grid ,semithick, Vertex/.style ={circle, inner sep=1pt, fill}]
\node[Vertex] (A1){};\node[Vertex] (A2) [below =of A1] {};\node[Vertex] (A3) [below =of A2] {};\node[Vertex] (A4)[below =of A3] {};
\node[Vertex] (B1)[right =of A1] {};\node[Vertex] (B2) [below =of B1] {};\node[Vertex] (B3) [below =of B2] {};\node[Vertex] (B4)[below =of B3] {};
\draw (A1) to (B2) (A3) to (B3)  (A4) to (B1);
\draw[red, dotted] (A2) to (B4);\end{tikzpicture} \text{ and }
\begin {tikzpicture}[node distance =0.3cm and 0.8cm ,on grid ,semithick, Vertex/.style ={circle, inner sep=1pt, fill}]
\node[Vertex] (A1){};\node[Vertex] (A2) [below =of A1] {};\node[Vertex] (A3) [below =of A2] {};\node[Vertex] (A4)[below =of A3] {};
\node[Vertex] (B1)[right =of A1] {};\node[Vertex] (B2) [below =of B1] {};\node[Vertex] (B3) [below =of B2] {};\node[Vertex] (B4)[below =of B3] {};
\draw (A1) to (B2) (A3) to (B1)  (A4) to (B4);
\draw[red, dotted] (A2) to (B3);\end{tikzpicture}
\begin {tikzpicture}[node distance =0.3cm and 0.8cm ,on grid ,semithick, Vertex/.style ={circle, inner sep=1pt, fill}]
\node[Vertex] (A1){};\node[Vertex] (A2) [below =of A1] {};\node[Vertex] (A3) [below =of A2] {};\node[Vertex] (A4)[below =of A3] {};
\node[Vertex] (B1)[right =of A1] {};\node[Vertex] (B2) [below =of B1] {};\node[Vertex] (B3) [below =of B2] {};\node[Vertex] (B4)[below =of B3] {};
\draw (A2) to (B2)  (A3) to (B3) (A4) to (B1);
\draw[red, dotted] (A1) to (B4);\end{tikzpicture}
\ \text{ and } \ 
\begin {tikzpicture}[node distance =0.3cm and 0.8cm ,on grid ,semithick, Vertex/.style ={circle, inner sep=1pt, fill}]
\node[Vertex] (A1){};\node[Vertex] (A2) [below =of A1] {};\node[Vertex] (A3) [below =of A2] {};\node[Vertex] (A4)[below =of A3] {};
\node[Vertex] (B1)[right =of A1] {};\node[Vertex] (B2) [below =of B1] {};\node[Vertex] (B3) [below =of B2] {};\node[Vertex] (B4)[below =of B3] {};
\draw (A1) to (B1) (A3) to (B2)  (A4) to (B4);
\draw[red, dotted] (A2) to (B3);\end{tikzpicture}
\begin {tikzpicture}[node distance =0.3cm and 0.8cm ,on grid ,semithick, Vertex/.style ={circle, inner sep=1pt, fill}]
\node[Vertex] (A1){};\node[Vertex] (A2) [below =of A1] {};\node[Vertex] (A3) [below =of A2] {};\node[Vertex] (A4)[below =of A3] {};
\node[Vertex] (B1)[right =of A1] {};\node[Vertex] (B2) [below =of B1] {};\node[Vertex] (B3) [below =of B2] {};\node[Vertex] (B4)[below =of B3] {};
\draw (A2) to (B2)  (A3) to (B3) (A4) to (B1);
\draw[red, dotted] (A1) to (B4);\end{tikzpicture}\end{center}
\end{exa}

\subsection{The Jacobi identity and assignments with supervisions}\label{Sec:Jacobi.and.Assignment}
Tropical Jacobi tells us that either more than one set of bijections corresponding to optimal base value ($=\adj(A)^{\wedge k}_{J,I}$) of $k$ assignments with supervisions $I$ on $J$ or, it is sufficient to calculate the optimal bijection on $A[[n]-I, [n]-J]$.

Using the same method as  the proof of Theorem \ref{thm:JacobiGraph}, we get the following.

\begin{pro}\label{Prop:JacobiEqualityFindAssignment} If equality holds in the tropical Jacobi identity, then
 an optimal set of $k$ assignments with supervisions can be found in $\mathcal{O}(n^3)$ time.
\end{pro}

\if{
\todo[inline]{Currently written in normalised case, extend to general case.

A: see generelized case remark (\ref{nnrm}). it's not really relevant to graphs in my opinion.\\

S: If it is not relevant then why is it not easy to extend? It seems to me that you are right about 
renumbering and we can keep at least that part of your remark (although I
commented it out).  However, it is not obviously so about rescaling.. 

I now have to agree with you that introducing weights of loops is not so straightforward. 
Regarding Case~2, if we have a configuration of paths resembling figure "8" then what is peculiar is that 
after replacing cycles with loops we will have two loops in the node where the cycles intersect and one loop in
all other points, while replacing two bijections by two bijections introduces only one extra loop at each of these nodes..
}
}\fi

\proof  Assume $\tau:I^C\rightarrow J^C$ attains $A^{\wedge(n-k)}_{I^C, J^C}$.  Then $\tau$ is composed of elementary paths $P\in\mathcal{P}$ and cycles $C\in\mathcal{C}$.  Observe that the cycles in $\mathcal{C}$ can be assumed loops (given that we are working in the normalized case).

For each path we construct one assignment  and identify its supervised edge as follows:

Let $p=(i_1, j_1=i_2, j_2=i_3,\dots,j_{k-1}=i_t, j_t)$ where $i_1,\dots i_t\in I$ and $j_1,\dots, j_t\in J$.
The edge/assignment $(t(p), s(p))$ completes $p$ into a cycle, and we add loops $(r,r)$ for $r\in [n]\setminus\{i_1,\dots,i_k\}$ to obtain $\pi_p\in S_n$.  The supervised edge in $\pi_p$ is $(t(p), s(p))$.

Note that, if there are less than $k$ elementary paths in the decomposition of $\tau$, then the remaining permutations and supervised edges are respectively  identity permutations and loops  between unassigned vertices.  
Finally, observe that there will never be more than $k$ elementary paths. 
 The most computationally expensive part of this procedure is finding the bijection at the start, this is $\mathcal{O}((n-k)^3)$ by the Hungarian Method.
\endproof

\if{
\section{Powers and conjugation of weight matrices}
\todo[inline]
{
{\bf S:} What did you mean to do in this section, and how is it necessary here?
}

\section{Notes on Complexity}

For square matrices $A\in\mathbb{R}^{n\times n}$ the assignment problem is solvable in $\mathcal{O}(n^3)$ time by the Hungarian method (H.M.) \cite{}.

Here we calculate the following:

$$\adj(A)\in\rmax^{n\times n}: \ \adj(A)_{ij}=\bigoplus_{\sigma\in S_{n-1}} w(\sigma, A_{\{j\}^c,\{i\}^c})=\bigoplus_{\pi\in S_n: \pi(j)=i} \bigg(w(\pi, A)-A_{j,i}\bigg),$$
Calculating $\adj(A)$ is $\mathcal{O}(n^5)$ ($n^2$ entries each $\mathcal{O}((n-1)^3)$ by H.M.).

$$A^{\wedge k}\in\rmax^{{n\choose k} \times {n\choose k}}: \ A^{\wedge k}_{I,J}=\bigoplus_{\pi\in S_k} w(\pi, A_{I, J})$$

Calculating $A^{\wedge k}_{IJ}$ is $\mathcal{O}(k^3)$ by H.M.

Calculating $A^{\wedge k}$ is $\mathcal{O}({n\choose k}^2k^3)$ which is only polynomial for \underline{fixed} $k$.  

So calculating the value of both LHS and RHS of tropical Jacobi identity can be done in polynomial time (since we don't have to calculate the full compound matrix).

\section{Open Questions and Discussion}
\begin{itemize}
\item We optimise a series of assignments given that one edge is fixed in each permutation.  But these 'special' assignments ('head of team') cannot be prechosen:  it depends on the permutation of max weight in $per(adj(A))$.   Do we have any control over the cost of these 'special' assignments?  

\item Necessary and/or sufficient conditions on \underline{when} we have each of the cases described by the Theorem~\ref{thm:JacobiGraph} 

\item A \textbf{Latin square} of order $n$ is an $n\times n$ array where each cell contains exactly one of the symbols $1,\dots,n$ so that each symbol appears exactly once in each row and column.  So, a Latin square can be interpreted as an $(n,n)$-regular graph.  Is this interesting?  

\item  (known definitions, useful?):  a \textbf{factor} or a graph $G$ is a spanning subgraph.  A $k$-factor of $G$ is a spanning $k$-regular subgraph.  A $k$-factorization partitions the edges into disjoint $k$-factors.  A $1$-factor is a perfect matching.

\item  What does the tropical Jacobi identity tell us about the assignments?  Can it be used to calculate the optimum assignments quicker in some cases?  If it holds with equality then what are the permutations?  \textbf{Partial answer Sec \ref{Sec:Jacobi.and.Assignment}.}
\item Is there a bound on how far $\adj(A)_{IJ}^{\wedge k}$ can be from $A^{\wedge (n-k)}_{J^CI^C}$? \textbf{Partial answer was in Sec \ref{Sec:JacobiEquality} of the previous version.}
\item Quicker way to calculate the optimal permutations (without so many calls to the Hungarian Algorithm)?
\item Quicker way to calculate compound matrix?
\item Relationship between $\adj(A)^{\wedge k}$ and $\adj(A^{\wedge k})$?
\item  Is there any meaning/usefulness to counting the number of times an edge appears in the bijections of the adjoint?  In this example the identity edges appeared most frequently in the matrix of bijections, the edge $(1,2)$ appears in 5 bijections, $(4,1)$ in 3 and $(2,4)$ in 4....
\end{itemize}

}\fi

\bibliographystyle{alpha}

\bibliography{tropical}
\end{document}